\pdfoutput=1
\RequirePackage{ifpdf}
\ifpdf % We~are running pdfTeX in pdf mode
\documentclass[pdftex]{sigma}%,draft
\else
\documentclass{sigma}
\fi

\usepackage{mathdots}
\usepackage{mathtools}
\RequirePackage{tikz}

\numberwithin{equation}{section}

\newtheorem{Theorem}{Theorem}[section]
\newtheorem*{Theorem*}{Theorem}
\newtheorem{Corollary}[Theorem]{Corollary}
\newtheorem{Lemma}[Theorem]{Lemma}
\newtheorem{Proposition}[Theorem]{Proposition}
{\theoremstyle{definition}
\newtheorem{Definition}[Theorem]{Definition}
\newtheorem{Example}[Theorem]{Example}
\newtheorem{Remark}[Theorem]{Remark}}

\renewcommand{\Re}{\operatorname{Re}}
\renewcommand{\Im}{\operatorname{Im}}

\DeclareMathOperator{\per}{\mathrm{per}}
\DeclareMathOperator{\id}{\mathrm{id}}
\DeclareMathOperator{\sv}{\mathrm{sv}}

\newcommand{\cP}{\mathcal{P}}

\newcommand{\mot}[0]{\mathfrak{m}}
\newcommand{\dr}[0]{\mathfrak{dr}}

\DeclareMathOperator{\li}{Li}
\newcommand{\Isv}{I^{\mathrm{sv}}}
\newcommand{\Idr}{I^{\mathfrak{dr}}}
\newcommand{\Icsv}{I^{\mathrm{csv}}}
\newcommand{\licsv}{\li^{\mathrm{csv}}}
\newcommand{\Imot}{I^{\mathfrak{m}}}
\newcommand{\Deltatilde}{\widetilde{\Delta}}
\newcommand{\Stilde}{\widetilde{S}}

\newcommand{\cPmpl}[0]{\cP^{\mathfrak{m}}_{\text{MPL}}}
\newcommand{\cPm}[0]{\cP^{\mathfrak{m}}}
\newcommand{\cPdrpl}[0]{\cP^{\mathfrak{dr}}_{\text{MPL}}}
\newcommand{\cPdr}[0]{\cP^{\mathfrak{dr}}}

\DeclareMathOperator{\svmot}{\sv^{\mathfrak{m}}}

\DeclareMathOperator{\cR}{\mathcal{R}}

\DeclareMathOperator{\Cl}{Cl}

\newcommand{\ber}[2]{\mathcal{B}_{#1}(#2)}

\begin{document}
\allowdisplaybreaks

\newcommand{\arXivNumber}{2104.04344}

\renewcommand{\thefootnote}{}

\renewcommand{\PaperNumber}{107}

\FirstPageHeading

\ShortArticleName{Clean Single-Valued Polylogarithms}

\ArticleName{Clean Single-Valued Polylogarithms\footnote{This paper is a~contribution to the Special Issue on Algebraic Structures in Perturbative Quantum Field Theory in honor of Dirk Kreimer for his 60th birthday. The~full collection is available at \href{https://www.emis.de/journals/SIGMA/Kreimer.html}{https://www.emis.de/journals/SIGMA/Kreimer.html}}}

\Author{Steven CHARLTON~$^{\rm a}$, Claude DUHR~$^{\rm b}$ and Herbert GANGL~$^{\rm c}$}

\AuthorNameForHeading{S.~Charlton, C.~Duhr and H.~Gangl}

\Address{$^{\rm a)}$~Fachbereich Mathematik (AZ), Universit\"at Hamburg, Bundesstra\textup{\ss}e 55, \\
\hphantom{$^{\rm a)}$} 20146 Hamburg, Germany}
\EmailD{\href{mailto:steven.charlton@uni-hamburg.de}{steven.charlton@uni-hamburg.de}}

\Address{$^{\rm b)}$~Bethe Center for Theoretical Physics, Universit\"at Bonn, 53115 Bonn, Germany}
\EmailD{\href{mailto:cduhr@uni-bonn.de}{cduhr@uni-bonn.de}}

\Address{$^{\rm c)}$~Department of Mathematical Sciences, Durham University, Durham DH1 3LE, UK}
\EmailD{\href{mailto:herbert.gangl@durham.ac.uk}{herbert.gangl@durham.ac.uk}}

\ArticleDates{Received April 13, 2021, in final form November 28, 2021; Published online December 12, 2021}

\Abstract{We define a variant of real-analytic polylogarithms that are single-valued and that satisfy ``clean'' functional relations that do not involve any products of lower weight functions.	We discuss the basic properties of these functions and, for depths one and two, we present some explicit formulas and results. We also give explicit formulas for the single-valued and clean single-valued version attached to the Nielsen polylogarithms $S_{n,2}(x)$, and we show how the clean single-valued functions give new evaluations of multiple polylogarithms at certain algebraic points.}


\Classification{11G55; 11M32; 33E20; 39B32}

\renewcommand{\thefootnote}{\arabic{footnote}}
\setcounter{footnote}{0}

\section{Introduction}
\subsection{Background and first definitions}\label{sec:definitions}

The logarithm function and generalisations of it have originally been studied, having first been mentioned (in 1696) in correspondence between (Johann) Bernoulli and Leibniz \cite[p.~351]{Leibniz:1696}, by many mathematicians, notably by Abel \cite{Abel_Oeuvres} and Kummer \cite{Kummer,Kummer+,Kummer++} with regard to their functional properties, and by Lobachevsky \cite{Lobatschewsky:1837} and later by Schl\"afli in connection with volume functions in hyperbolic space (for a far more comprehensive list of the early bibliography see Lewin's book \cite[pp.~349--353]{LewinBook}).
Over the last 40--50 years, seminal works on the dilogarithm, pioneered by Bloch \cite{Bloch_Irvine} in algebraic geometry and algebraic K-theory and by 't~Hooft and Veltman \cite{tHooft:1972tcz} in connection with quantum field theory, have led to a renaissance of interest in those functions and have triggered many new and often unexpected and surprisingly parallel developments, resulting in ``cross-fertilisation'' from which both mathematics (keyword ``mixed motives (over a field)'') and physics (keyword ``Feynman integrals'') communities have benefited.

The logarithm is a complex multi-valued function on $\mathbb{C}\setminus\{0\}$, and it can be defined on its principal branch $\mathbb{C}\setminus (-\infty,0]$ by the integral
\begin{gather*}
\log x := \int_{1}^x\frac{{\rm d}t}{t} , \qquad x\in \mathbb{C}\setminus (-\infty,0] .
\end{gather*}
The most prominent generalisations of the logarithm function are the so-called \emph{classical poly\-lo\-ga\-rithms}, defined for integers $n>0$ by
\begin{gather*}
\li_n(x) := \sum_{k=1}^\infty \frac{x^k}{k^n} .
\end{gather*}
The integer $n$ is called the \emph{weight}.
The series converges for $|x|<1$. It can be analytically continued to a multi-valued function on the whole complex plane via the integral representation
\begin{gather*}
\li_n(x) = \int_{0}^x\li_{n-1}(t) \frac{{\rm d}t}{t} ,\qquad n>1 ,
\end{gather*}
and the recursion starts with $\li_1(x)=-\log(1-x)$.
Classical polylogarithms are not rich enough to cover all the generalisations of the logarithm that appear in mathematics and physics. A~broader class of generalisations of the logarithm function are \emph{multiple polylogarithms} (MPL's) (also known as hyperlogarithms), which were first introduced in the works of Poincar\'e, Kummer and Lappo-Danilevsky~\cite{Kummer,Kummer+,Kummer++,Lappo:1927} and have recently reappeared in both mathematics~\cite{ChenSymbol,GoncharovMixedTate,Goncharov:1998kja} and physics~\cite{Ablinger:2011te,Gehrmann:2000zt,Remiddi:1999ew}. Multiple polylogarithms can be defined by the iterated integral
\begin{gather}\label{eq:G_def}
I(x_0;x_1,\ldots,x_n;x_{n+1}) := \int_{x_0<t_1<\cdots<t_n<x_{n+1}}\frac{{\rm d}t_1}{t_1-x_1}\wedge\cdots\wedge \frac{{\rm d}t_n}{t_n-x_n} ,
\end{gather}
where $x_i\in\mathbb{C}$. The integer $n$ is again called the weight, and the number of non-zero elements of $(x_1,\ldots,x_n)$ is called the {depth}. The notation
\begin{gather*}
 I_{m_1,\ldots,m_k}(x_1,\ldots,x_k) := I\big(0;x_1,\{0\}^{m_1-1},\ldots,x_{k},\{0\}^{m_k-1};1\big)
 \end{gather*}
 is often employed to write a depth $k$ integral, where $\{a\}^n$ denotes $a$ repeated $n$ times. The integral implicitly depends on the choice of a path going from $x_0$ to $x_{n+1}$, where the integration variables $x_i$ are considered to be ordered on the path. Depending on the values of the $x_i$, the integral in~\eqref{eq:G_def} may diverge and requires regularisation. This can be done by introducing suitable tangential base points, cf.~\cite[Chapter~15]{Deligne:1989}.
The class of functions defined by~\eqref{eq:G_def} contains the logarithm and classical polylogarithm functions as special cases, e.g., for generic values of $x_0$, $x_1$, $x_2$,
\begin{gather*}
I(x_0;x_1;x_2) = \log\bigg(\frac{x_1-x_2}{x_1-x_0}\bigg),\\
 I\big(0;1,\{0\}^{n};x_0\big) = -\li_{n+1}(x_0) .
\end{gather*}

The definition of MPL's in~\eqref{eq:G_def} implies that they satisfy the following basic relations common to all iterated integrals (cf., e.g.,~\cite{ChenSymbol}):
\begin{enumerate}\itemsep=0pt
\item \emph{Path reversal} (here we assume the same path for both, except that it is being traversed in opposite directions):
\begin{gather*}
I(x_{n+1};x_n,\ldots,x_1;x_0) = (-1)^n I(x_0;x_1,\ldots,x_n;x_{n+1}) .
\end{gather*}
\item \emph{Path composition} (for any $x\in \mathbb{C}$ and any path from $x_0$ to $x_{n+1}$ avoiding any $x_i$ ($1\leq i\leq n$)):
\begin{gather*}
I(x_0;x_1,\ldots,x_n;x_{n+1}) = \sum_{p=0}^nI(x_0;x_1,\ldots,x_p;x) I(x;x_{p+1},\ldots,x_n;x_{n+1}) .
\end{gather*}
\item \emph{Shuffle product} (in this equality we assume the same path for each iterated integral):
\begin{gather*}
I(x_0;x_{1},\ldots,x_m;x)I(x_0;x_{m+1},\ldots,x_{m+n};x)
 = \!\!\sum_{\sigma\in\Sigma(m,n)}\!I(x_0;x_{\sigma(1)},\ldots,x_{\sigma(m+n)};x) ,
\end{gather*}
where $\Sigma(m,n) = \big\{\sigma\in S_{m+n}\colon \sigma^{-1}(1)<\!\cdots\!<\sigma^{-1}(m)\text{ and } \sigma^{-1}(m+1)<\!\cdots\!<\sigma^{-1}(m+n)\big\}$ is the set of shuffles of $m$ and $n$ elements, and $S_{m+n}$ is the group of permutations on $m+n$ elements.
\end{enumerate}

\subsection{Identities among polylogarithms}
The identities at the end of the previous section are special, in the sense that they relate many different MPL's evaluated at the same arguments, albeit in a different order. More interesting are identities involving a single type of function evaluated at different arguments. The most famous identity involving dilogarithms is arguably the five-term relation due to Abel (cf., e.g.,~\cite[Chapter~1.5]{LewinBook}), a version of which where the order of the five arguments for $\li_2(z)$ defines a~5-cycle ($1-z_i = z_{i+2}z_{i-2}$, indices mod 5) {being} given by
\begin{gather}
\nonumber\li_2(x) + \li_2(y) +\li_2\bigg(\frac{1-x}{1-xy}\bigg) + \li_2(1-xy) +\li_2\bigg(\frac{1-y}{1-xy}\bigg)
\\ \qquad
{} =\zeta_2-\log x \log(1-x)-\log y \log(1-y)+\log\bigg(\frac{1-x}{1-xy}\bigg) \log\bigg(\frac{1-y}{1-xy}\bigg) ,
\label{eq:five-term}
\end{gather}
with $\zeta_n:=\li_n(1)$.
Since the logarithm and dilogarithm are multi-valued functions, it is important to specify the branches of the functions and the ranges for $x$ and $y$ for which this identity holds. It is straightforward to check that on the principal branches of the logarithm {(branch cut from $\-\infty$ to $0$)} and dilogarithm {(branch cut from 1 to $\infty$)} the identity in~\eqref{eq:five-term} holds whenever $|x|+|y|<1$. It has often been claimed in the literature, without explicit proof, as a kind of ``folklore'' statement, that every functional equation for $\li_2$ with arguments being rational functions in finitely many variables is a linear combination of this five-term relation. Wojtkowiak proved it for the 1-variable case \cite{Wojtkowiak:1996}, and for a recent proof of the general statement we refer to a recent preprint by de Jeu \cite{deJeu}.
%with $\zeta_n:=\li_n(1)$.

Seminal non-trivial identities involving logarithms and classical polylogarithms of higher weight have been found, e.g.,~by Kummer \cite{Kummer,Kummer+,Kummer++} (in two variables, up to weight 5), Goncharov (in three variables, weight 3) \cite{Goncharov:1995}, Wojtkowiak (in many variables, weight 3) \cite{Wojtkowiak:1991} and by Gangl (in two variables, up to weight~7 \cite{Gangl:1995, Gangl:2003}; in four variables, weight 4 \cite{Gangl_weight4}), as well as many others, and particularly interesting recent findings are given by Golden, Goncharov, Spradlin, Vergu and Volovich \cite{Golden:2013xva}, Charlton \cite{CharltonPhD}, Radchenko \cite{RadchenkoPhD} and Goncharov--Rudenko \cite{GoncharovRudenko}. No such results beyond weight 7 are currently known.
There are also families of (sometimes called ``tri\-vial'') identities in one variable known for all weights relating a specific classical polylogarithm at different arguments, and possibly products of logarithms, in particular the \emph{distribution relations}, valid as power series in the unit disk
\begin{gather}\label{eq:distribution}
\li_n(x^m) = m^{n-1}\sum_{k=0}^{m-1}\li_n\hskip -3pt \big(x\xi_m^k\big),\qquad |x|<1,\quad \xi_m^m=1 ,
\end{gather}
 and the \emph{inversion relations},
\begin{gather}\label{eq:inversion}
\li_{n}\bigg(\frac1x\bigg) = (-1)^{n+1} \li_n(x) -\frac{(-2\pi {\rm i})^n}{n!} B_n\bigg(\frac{1}{2}+\frac{\log(-x)}{2\pi {\rm i}}\bigg) ,
\end{gather}
where $x\in\mathbb{C}\setminus[0,\infty)$ and $B_n(\alpha)$ are the Bernoulli polynomials, defined by the generating series
\begin{gather}\label{eqn:bernoulli}
\frac{t {\rm e}^{\alpha t}}{{\rm e}^t-1} =\sum_{n=0}^\infty B_n(\alpha) \frac{t^n}{n!} .
\end{gather}

Much less is known about identities satisfied by MPL's of depth greater than one, { although the basic shuffle and stuffle relations were established by Goncharov in \cite{GoncharovMixedTate}, along with a gene\-ralisation of the distribution relations to any fixed MPL of depth greater than one, and an ``inversion-reversion'' relation \cite[Section~2.6, formulas (33) and (34)]{GoncharovMixedTate} valid on the unit $m$-torus.}
The study of { MPL identities of depth greater than one} has recently obtained new impetus from physics, where MPL's and their identities play an important role in the computation of scattering amplitudes in quantum field theory, cf., e.g.,~\cite{Ablinger:2018sat,Ablinger:2011te,Ablinger:2013cf,Duhr:2012fh,Duhr:2014woa,Duhr:2011zq,
Gehrmann:2001pz,Gehrmann:2000zt,Gehrmann:2001jv,
panzer_parity,Remiddi:1999ew,Vollinga:2004sn}. Goncharov \cite{Goncharov:2005sla} introduced the arguably most important invariant for multiple polylogarithms, its ``symbol''. Based on techniques to compute it, developed originally for quantum field theory calculations \cite{Spradlin:2011wp}, new functional identities for polylogarithms of different depth have been found,
for example: a 40-term trilogarithm identity whose arguments arise from a single cluster algebra is obtained in \cite{Golden:2013xva}; a new family of functional equations for $\li_4$ are given in \cite{ganglSome2000}, based on a~depth reduction in weight~4; various relations between weight 4 MPL's of any depths are given in \cite{Gangl_weight4}, including a reduction of a certain 5-term combination of $I_{3,1}$ to depth~1, from which a highly symmetric 4-variable $\li_4$ functional equation is obtained. Various relations between weight 5 MPL's of any depths are given in~\cite{CharltonPhD}, including a reduction of $I_{3,2}$ to $I_{4,1}$ and $\li_5$ terms, and an explicit inversion result relating $I_{a,b}(x,y)$ and $I_{a,b}\big(x^{-1},y^{-1}\big)$ for any depth 2 MPL. Concurrently an inversion result valid for {an} MPL of {arbitrary} depth was given in \cite{panzer_parity}, a~clean single-valued version of which (up to depth 3) we {provide} in Section~\ref{sec:examples}. Further reductions in weight~4 and~5, focusing on the so-called Grassmannian polylogarithm, are investigated in~\cite{Charlton:2019ilv}, whereas identities and reductions involving the so-called Nielsen polylogarithms in weights 5 through 8 are investigated in \cite{Charlton:2019gvp} (also using the clean single-valued version established in~Section~\ref{sec:nielsen} below).

\subsection{Clean single-valued polylogarithms and their identities}
As already mentioned, the multi-valuedness of MPL's implies that identities among them are to be understood as holding on appropriate branches. In order to circumvent this cumbersome issue, it is useful to replace {any} MPL's by a version {of it} that, while only real-analytic, has the virtue that it is single-valued. For example, the original single-valued version of the dilogarithm was given by Bloch and by Wigner\footnote{This is the mathematician David Wigner, as opposed to the arguably better known physicist Eugene (incidentally his father).} \cite{Bloch_Irvine} and generalised to polylogarithms (implicitly) by Ramakrishnan \cite{Ramakrishnan:1986} and (explicitly) by Wojtkowiak \cite{Wojtkowiak:1989} and Zagier \cite{Zagier:1991}. {The latter author proposed in fact several versions, the most standard one being} defined as\vspace{-1ex}
\begin{gather}\label{eq:P_def}
P_n(x) := \mathfrak{R}_n\Bigg\{\sum_{k=0}^{n-1}\frac{2^k B_k}{k!} \log^k|x| \li_{n-k}(x)\Bigg\} ,
\end{gather}
where\vspace{-1ex}
\begin{gather}\label{eq:Rn_def}
\mathfrak{R}_n = \begin{cases}
\Re ,& \text{if}\ n \text{ odd},
\\
\Im ,& \text{if}\ n \text{ even},
\end{cases}
\end{gather}
and Re and Im denote the real and imaginary parts respectively.
Moreover, $B_k:=B_k(0)$, are the Bernoulli numbers, defined as the constant terms of the Bernoulli polynomials defined above.
A~rather different single-valued version was given by Brown
\cite{BrownSVHPLs,brownSV}. We will relate the two explicitly in Section~\ref{sec:sv}.
The functions in equation~\eqref{eq:P_def} satisfy a ``product-free'' variant of the five-term relation in~\eqref{eq:five-term}\vspace{-1ex}
\begin{gather}\label{eq:five-term-P}
P_2(x) + P_2(y) + P_2\bigg(\frac{1-x}{1-xy}\bigg) + P_2(1-xy)+P_2\bigg(\frac{1-y}{1-xy}\bigg) =0 ,\vspace{-1ex}
\end{gather}
and of the distribution and inversion relations in~\eqref{eq:distribution} and~\eqref{eq:inversion},\vspace{-1ex}
\begin{gather}
\begin{split}
&P_n\big(x^m\big) = m^{n-1}\sum_{k=0}^{m-1}P_n\big(x\xi_m^k\big),
\\
& P_n\big(x^{-1}\big) = (-1)^{n+1} P_n(x) ,\qquad n>1 .
\end{split}\label{eq:dist-inv-P}
\end{gather}

\noindent
{\samepage
Since the functions $P_n(z)$ are single-valued, the identities in~\eqref{eq:dist-inv-P} are valid for all complex numbers $x\not=0$, while the five-term relation in~\eqref{eq:five-term-P} holds for $(x,y)\in\mathbb{C}^2\setminus L$, where $L$ is the union of curves defined by $x=0$, $x=1$, $y=0$, $y=1$ and $xy=1$.

}

The identities in~\eqref{eq:five-term-P} and~\eqref{eq:dist-inv-P} have an additional feature compared to their analogues in~\eqref{eq:five-term}, \eqref{eq:distribution} and~\eqref{eq:inversion}: they do not involve product terms of \mbox{(poly-)}logarithms of lower weights! We refer to an identity with this property, in line with standard terminology, e.g., \cite[Section~6]{Zagier:1991}, as a \emph{clean} identity. More generally, roughly stated for every identity involving classical polylogarithms of weight $\leq n$
we can obtain a clean identity by replacing $\li_n$ by $P_n$ and dropping all product terms. For the precise statement we refer to \cite[Propositions~2 and~3]{Zagier:1991}.
For MPL's of higher depths, however, in general no real-analytic analogues are known that satisfy clean versions of identities between the iterated integrals in~\eqref{eq:G_def}.
In the classical case, the product-freeness of relations permits one to mimic the functional behaviour via rather simple general (linear and multi-linear) algebraic tools, more precisely of quotients of free abelian groups like the so-called higher Bloch groups. For the latter groups the relations arise from taking only the non-product terms in a functional equation for $\li_n$, i.e., non-linear contributions are simply being ignored. In a similar way, one might hope that the clean functions give rise to ``simpler'' higher depth analogues of said Bloch groups, without the need to consider products of lower weight terms. One of the main results of this paper is to define such functions for all weights and depths. In the remainder of this section we summarise our main result.

It is possible to lift the iterated integrals $I(x_{0};x_{1},\ldots,x_{n};x_{n+1})$, for $x_i\in\overline{\mathbb{Q}}$, to motivic versions $\Imot(x_{0};x_{1},\ldots,x_{n};x_{n+1})$, which live in a ring of motivic periods $\cPmpl$ (see, e.g.,~\cite{brownmixedZ,Brown:2011ik,brownnotesmot, Goncharov:2005sla}). The ring $\cPmpl$ is graded by the weight of the MPL's, and we denote the subspace of weight $n$ by $\cPm_{\mathrm{MPL},n}$, and $\cPm_{\mathrm{MPL},>0}:= \bigoplus_{n>0}\cPm_{\mathrm{MPL},n}$. The iterated integrals in~\eqref{eq:G_def} can be retrieved from their motivic avatars through the period homomorphism $\per\colon\cPmpl\to \mathbb{C}$, which is conjectured to be injective~\cite{Grothendieck:1966}.
Therefore, it is expected that all relations among MPL's arise from relations among their motivic avatars. Within the motivic setting, we prove the following result in~Section~\ref{sec:clean}:\looseness=1
\begin{Theorem}\label{thm:main}
For every $\Imot(x_0;x_1,\ldots,x_n;x_{n+1})$ there is a real-analytic single-valued function $\Icsv(x_0;x_1,\ldots,x_n;x_{n+1})$ such that for every linear combination of motivic MPL's that can be reduced to products,
\begin{gather*}
\sum_{k=1}^Kc_k \Imot(x_{k,0};x_{k,1},\ldots,x_{k,n};x_{k,n+1}) \in \cPm_{\mathrm{MPL},>0}\cdot \cPm_{\mathrm{MPL},>0} ,\qquad c_k\in\mathbb{Q} ,
\end{gather*}
there is a clean identity {where products are mapped to zero in going from $\Imot$ to $\Icsv$, i.e., }
\begin{gather*}
\sum_{k=1}^Kc_k \Icsv(x_{k,0};x_{k,1},\ldots,x_{k,n};x_{k,n+1}) = 0 .
\end{gather*}
\end{Theorem}

The paper is structured as follows: In Section~\ref{sec:hopf} we review some basic facts about graded and connected Hopf algebras and the Dynkin operator. In Section~\ref{sec:motivic} we review the Hopf algebra on (de Rham) multiple polylogarithms, and we introduce the single-valued projection, which assigns to every multiple polylogarithm a real-analytic single-valued analogue. In Section~\ref{sec:clean} we define the clean version of single-valued multiple polylogarithms, and we discuss their basic properties. In particular, we show that they satisfy Theorem~\ref{thm:main}. In Section~\ref{sec:examples} we present some examples of clean single-valued MPL's in depths 1 and 2, and in Section~\ref{sec:nielsen} we explicitly compute the single-valued and clean single-valued versions of the Nielsen polylogarithm $S_{n,2}$. Finally in Section~\ref{sec:numerical} we derive some explicit numerical evaluations of depth 2 MPL's using this machinery and some known functional equations.

\section{Graded connected Hopf algebras and the Dynkin operator}
\label{sec:hopf}
This section reviews material from~\cite{Kreimer_Dynkin,Patras,Reutenauer,Reutenauer_Book}.
Let $H$ be a graded connected commutative Hopf algebra over $\mathbb{Q}$. The counit is simply the augmentation map $\epsilon\colon H\to H_0\simeq \mathbb{Q}$, and we have a splitting
\begin{gather*}
H = H_0 \oplus H_{>0},\qquad\text{with}\quad H_{>0} := \ker \epsilon .
\end{gather*}
The multiplication in $H$ is denoted by $m$ and the coproduct by $\Delta$. For $x\in H_{>0}$ it takes the form
\begin{gather*}
\Delta(x) = 1\otimes x+x\otimes 1 + \Delta'(x) , \qquad \Delta'(x) \in H_{>0}\otimes H_{>0} .
\end{gather*}
The antipode for $x\in H_{>0}$ is uniquely determined in a recursive way by
\begin{gather*}
0 = m(\id\otimes S)\Delta(x) = S(x) + x + m(\id\otimes S)\Delta'(x) .
\end{gather*}

\subsection{The convolution product}
Let $R$ be a unital $\mathbb{Q}$-algebra, with multiplication $m_R$ and unit $u_R\colon \mathbb{Q}\to R$. Let $\varphi,\psi\colon H\to R$ be $\mathbb{Q}$-linear maps. Their convolution is the $\mathbb{Q}$-linear map
\begin{gather*}
\varphi\star \psi := m_R(\varphi \otimes \psi) \Delta .
\end{gather*}
The co-associativity of $\Delta$ implies associativity of the convolution product. The set of all $\mathbb{Q}$-linear maps from $H$ to $R$ equipped with the convolution product forms a unital $\mathbb{Q}$-algebra, whose unit is $u_R\epsilon\colon H\to R$. Moreover, if $\varphi\colon H\to R$ is an algebra morphism, then it is invertible for the convolution product, and the inverse is simply composition with the antipode, $\varphi^{\star-1}:=\varphi S$. In~particular, the antipode is the inverse of the identity for the convolution product, $\id^{\star-1}=S$.

\begin{Definition} We say that a linear map $\varphi\colon H\to H$ is:
\begin{enumerate}\itemsep=-1pt
\item[$(1)$] a \emph{derivation}, if it satisfies $\varphi m = m(\varphi\otimes \id + \id\otimes \varphi)$,
\item[$(2)$] a \emph{co-derivation}, if it satisfies $\Delta\varphi = (\varphi\otimes \id + \id\otimes \varphi)\Delta$,
\item[$(3)$] an \emph{infinitesimal character}, if it satisfies $\varphi m = m(\varphi\otimes \epsilon + \epsilon\otimes \varphi)$.
\end{enumerate}
\end{Definition}

\begin{Lemma}\label{lem:inf_char}
Let $H$ be a graded, connected, commutative Hopf algebra and let $\varphi\colon H\to H$ be a~deri\-vation. Then $S\star \varphi$ is an infinitesimal character.
\end{Lemma}

\begin{proof}
We denote by $\tau\colon H\otimes H\to H\otimes H$ the operator that swaps the factors in a tensor product, $\tau(a\otimes b) = b\otimes a$. We have
\begin{align*}
(S\star\varphi)m & = m(S\otimes \varphi)\Delta m= m((Sm)\otimes (\varphi m))(\id\otimes \tau\otimes \id)(\Delta\otimes\Delta)\\
& =m(m\otimes m) (S\otimes S\otimes \varphi\otimes \id + S\otimes S\otimes\id \otimes \varphi)(\id\otimes \tau\otimes \id)(\Delta\otimes\Delta)\\
& =m(m\otimes m) (\id\otimes \tau\otimes \id) (S\otimes \varphi \otimes S\otimes \id + S\otimes \id \otimes S\otimes \varphi)(\Delta\otimes\Delta)\\
& =m((S\star\varphi)\otimes \epsilon + \epsilon\otimes (S\star\varphi)) . \tag*{\qed}
\end{align*}
\renewcommand{\qed}{}
\end{proof}

\subsection{The grading operator and the Dynkin operator}
\label{sec:Dynkin}

On every graded connected Hopf algebra there is a natural \emph{grading operator} $Y\colon H\to H$ which acts on homogeneous elements by multiplication by the weight. It is both a derivation and a~co-derivation:
\begin{gather*}
Ym = m(Y\otimes \id+\id\otimes Y) ,
\\
\Delta Y = (Y\otimes \id+\id\otimes Y)\Delta .
\end{gather*}

{We now introduce the \emph{Dynkin operator} $D$ on a graded connected Hopf algebra. For the origin of the name, see~\cite{Reutenauer} and references therein.}
\begin{Definition}
The Dynkin operator on $H$ is defined by
\begin{gather*}
D := S\star Y .
\end{gather*}
\end{Definition}
Since $\id \star S= \epsilon$, we can write the previous equation in the equivalent form
\begin{gather}\label{eq:id_S_D}
\id \star D = Y .
\end{gather}
Since $Y$ is a derivation, $D$ is an infinitesimal character by Lemma~\ref{lem:inf_char}.
It is convenient to define the operator $\Pi$ which is the identity on $H_0$ and $\Pi = Y^{-1} D$ on $H_{>0}$.

\begin{Proposition}\label{prop:ker}$\phantom{A}$
\begin{enumerate}\itemsep=0pt
\item[$1.$] The kernel of $\Pi$ is generated by all non-trivial products, $\ker \Pi = H_{>0}\cdot H_{>0}$.
\item[$2.$] $\Pi$ is a projector, $\Pi^2=\Pi$.
\end{enumerate}
\end{Proposition}

\begin{proof}
(1) Since $D$ is an infinitesimal character, we have for all $x,y\in H_{>0} = \ker \epsilon$,
\begin{gather*}
D(x\cdot y) = D(x)\cdot\epsilon(y) + \epsilon(x)\cdot D(y) = 0 ,
\end{gather*}
and so $\Pi(x\cdot y)=0$. Hence $H_{>0}\cdot H_{>0}\subset \ker \Pi$.
Conversely, let $x\in \ker \Pi$. We can assume without loss of generality that $x\in H_n$, $n>1$. Again writing the
coproduct as $\Delta(x) = 1\otimes x + x\otimes 1 + \Delta'(x)$, we find
\begin{gather*}
0 = \Pi(x) = x + \frac{1}{n} m(S\otimes Y)\Delta'(x) ,
\end{gather*}
and so
\begin{gather*}
x = -\frac{1}{n} m(S\otimes Y)\Delta'(x)\in H_{>0}\cdot H_{>0} .
\end{gather*}

(2) If $x\in H_n$, $n>0$, we have
\begin{gather*}
\Pi^2(x) = \frac{1}{n} \Pi\left[n x + m(S\otimes Y)\Delta'(x)\right] = \Pi(x) . \tag*{\qed}
\end{gather*}
\renewcommand{\qed}{}
\end{proof}

\section{Review of motivic polylogarithms}
\label{sec:motivic}

\subsection{Motivic and de Rham periods}

In~\cite{brownmixedZ,Brown:2011ik,brownnotesmot}, Brown has shown how to lift the iterated integrals $I(x_{0};x_{1},\ldots,x_{n};x_{n+1})$, for \mbox{$x_i\in\overline{\mathbb{Q}}$}, to motivic versions $\Imot(x_{0};x_{1},\ldots,x_{n};x_{n+1})$. The motivic MPL's generate a subring $\cPmpl$ inside the ring of all motivic periods $\cPm$.\footnote{Strictly speaking, $\Imot(x_{0};x_{1},\ldots,x_{n};x_{n+1})$ defines a \emph{family} of motivic periods depending on the variables $x_i$, see~\cite[Section~7]{brownnotesmot}. Since no confusion arises, we will always simply refer to $\Imot(x_{0};x_{1},\ldots,x_{n};x_{n+1})$ as a motivic period.} {A detailed review of the definition and construction of motivic and de Rham MPLs would go beyond the scope of this paper.
We refer, e.g., to~\cite[Section~10.6]{brownnotesmot}, for the example of the classical polylogarithms.} There is a natural homomorphism, called the \emph{period map}, $\per\colon \cPm\to \mathbb{C}$ such that
\begin{gather*}
\per\big(\Imot(x_{0};x_{1},\ldots,x_{n};x_{n+1})\big) = I(x_{0};x_{1},\ldots,x_{n};x_{n+1}) .
\end{gather*}
It follows from Grothendieck's period conjecture that $\per$ is expected to be injective.

The motivic MPL's are equipped with additional structure with respect to their non-motivic counterparts. In particular, they are equipped with a coaction
\begin{gather*}
\Delta\colon\ \cPmpl\to \cPmpl\otimes \cPdrpl ,
\end{gather*}
given on motivic MPL's via the formula~\cite{brownmixedZ,Brown:2011ik}
\begin{gather}
\Delta\big(\Imot(x_{0};x_{1},\ldots,x_{n};x_{n+1})\big) \nonumber
\\ \qquad
{}=\!\!\sum_{\substack{0=i_0<i_1<\cdots\\<i_k<i_{k+1}=n+1}}
\!\!\Imot(x_{0};x_{i_1},\ldots,x_{i_k};x_{n+1})\otimes \prod_{p=0}^k\Idr(x_{i_p};x_{i_p+1},\ldots,x_{i_{p+1}-1};x_{i_{p+1}}) .
\label{eq:coaction}
\end{gather}
The quantities in the second factor of the tensor product are de Rham MPL's and those span the ring $\cPdrpl$, which can be thought of as the quotient of $\cPmpl$ by the ideal generated by $(2\pi {\rm i})^{\mathfrak{m}}$ (the motivic lift of $2\pi {\rm i}$). There is a natural projection (see, e.g., \cite[Section~4.3]{brownnotesmot}):
\begin{gather*}
\pi^{\mathfrak{dr}}\big(\Imot(x_{0};x_{1},\ldots,x_{n};x_{n+1})\big) = \Idr(x_{0};x_{1},\ldots,x_{n};x_{n+1}) .
\end{gather*}
 $\cPdrpl$ is a commutative connected Hopf algebra graded by the weight (where the weight of $\Idr(x_{0};x_{1},\ldots,x_{n};x_{n+1})$ is defined as $n$). The coproduct on $\cPdrpl$ is given by the same formula as in~\eqref{eq:coaction}, with $\Imot$ replaced by $\Idr$ everywhere~\cite{brownmixedZ,Brown:2011ik,brownnotesmot, Goncharov:2005sla},\footnote{The ``motivic'' MPL's defined by Goncharov in~\cite{Goncharov:2005sla} correspond to the de Rham MPL's defined by Brown in~\cite{brownmixedZ,Brown:2011ik,brownnotesmot}. Here we consistently follow Brown's nomenclature.} and we use the same symbol to denote the coaction on $\cPmpl$ and the coproduct on $\cPdrpl$. Since $\cPdrpl$ is graded and connected, the antipode $S$ is uniquely determined by the coproduct.

\subsection{Single-valued projection}
\label{sec:sv}
Unlike motivic MPL's, to which the period map assigns a (complex) number, de Rham MPL's do not allow for an analogous construction. Instead, they can be equipped with a ring homomorphism $\sv\colon \cPdrpl\to\cPmpl$, called the \emph{single-valued projection} (cf.~\cite[Section~8.3]{brownnotesmot}, and also~\cite{BrownSVHPLs,brownSV}). The single-valued projection can be given explicitly in a combinatorial way on $\cPdrpl$~\cite{Brown:2013gia, brownnotesmot} (see also~\cite[Section~3.4]{DelDuca:2016lad}),
\begin{gather*}
\sv := m(F_{\infty}\Sigma\otimes \id)\Deltatilde ,
\end{gather*}
where $m$ is the multiplication in $\cPmpl$, $F_{\infty}\colon \cPmpl\to\cPmpl$ is the real Frobenius, which can be thought of as complex conjugation (i.e., $\per F_{\infty} = \overline{\per}$, where $\overline{\per}(x)$ denotes the complex conjugate of $\per(x)$) and $\Sigma\colon \cPmpl\to\cPmpl$ is defined by
\begin{gather*}
\Sigma\big(\Imot(x_{0};x_{1},\ldots,x_{n};x_{n+1})\big) := (-1)^n \Stilde\big(\Imot(x_{0};x_{1},\ldots,x_{n};x_{n+1})\big) .
\end{gather*}
Here $\Deltatilde\colon \cPdrpl\to \cPmpl\otimes\cPmpl$ and $\Stilde\colon \cPmpl\to\cPmpl$ are given by the same formulas as the coproduct $\Delta$ and the antipode $S$ on $\cPdrpl$, with the replacement $\Idr\to\Imot$ everywhere.

The single-valued projection associates to every $\Idr(x_{0};x_{1},\ldots,x_{n};x_{n+1})$ a (family of) motivic periods, whose image under the period map defines a single-valued function of the $x_i$. We can compose the single-valued projection with the period map and the projection $\pi^{\dr}$ to associate to every motivic MPL its single-valued version:
\begin{gather*}
\svmot:=\per\circ \sv \circ\pi^{\mathfrak{dr}} \colon\ \cPmpl\to \mathbb{C} .
\end{gather*}

\begin{Example}[single-valued version of the motivic logarithm]
We can apply the previous construction to the motivic logarithm
\begin{gather*}
\log^{\mot}x := \Imot(0;0;x) ,\qquad x\in \overline{\mathbb{Q}}\setminus\{0\} .
\end{gather*}
We have, with $\log^{\dr}x = \pi^{\dr}(\log^{\mot}x)$,
\begin{gather*}
\begin{split}
&\Delta\big(\log^{\dr}x\big) = \log^{\dr}x \otimes 1 + 1\otimes \log^{\dr}x,\\
&S\big(\log^{\dr}x\big) = -\log^{\dr}x .
\end{split}
\end{gather*}
The single-valued version attached to $\log^{\mot}x$ is therefore
\begin{gather*}
\svmot(\log^{\mot}x) = \log\overline{x} + \log x = \log|x|^2 .
\end{gather*}
In particular, letting $x=-1$, we see that the single-valued version attached to $(\pi {\rm i})^{\mot}$ is zero, i.e.,
\begin{gather*}
\svmot\big((\pi {\rm i})^{\mot}\big) = 0 .
\end{gather*}
\end{Example}

\begin{Example}[single-valued version of the classical motivic polylogarithm]
The motivic lift of the classical polylogarithm of weight $n$ is
\begin{gather*}
\li^{\mot}_n(x) := -\Imot\big(0;1,\{0\}^{n-1};x\big) , \qquad x\in\overline{\mathbb{Q}} .
\end{gather*}
The coproduct and the antipode of $\li^{\dr}_n(x) := \pi^{\dr}\big(\li^{\mot}_n(x)\big)$ are
\begin{gather}
\begin{split}
& \Delta\big(\li^{\dr}_n(x)\big) = \li^{\dr}_n(x)\otimes 1 + 1\otimes \li^{\dr}_n(x) + \sum_{k=1}^{n-1}\li^{\dr}_{n-k}(x)\otimes\frac{\log^{\dr}(x)^k}{k!},\\
& S\big(\li^{\dr}_n(x)\big) = -\li^{\dr}_n(x) - \sum_{k=1}^{n-1}\frac{(-\log^{\dr} x)^k}{k!}\li^{\dr}_{n-k}(x) .
\end{split}
\label{eq:Delta_Li}
\end{gather}
The single-valued version attached to the classical motivic polylogarithm of weight $n$ is then
\begin{gather}\label{eq:svLi}
\svmot\big(\li^{\mot}_n(x)\big)= \li_n(x) - (-1)^n\sum_{k=0}^{n-1}\frac{(-\log|x|^2)^k}{k!} \li_{n-k}(\overline{x}) .
\end{gather}
Letting $x=1$ in~\eqref{eq:svLi}, we obtain the single-valued version associated to the motivic zeta values, $\zeta_n^{\mot} := \li_n^{\mot}(1)$, $n>1$~\cite{Brown:2013gia}:
\begin{gather}\label{eq:sv_zeta}
\svmot(\zeta_{n}^{\mot}) = \svmot\big(\li_{n}^{\mot}(1)\big) =
\begin{cases}
2 \zeta_{2m+1}, & n\ \text{odd} ,
\\
0, & n\ \text{even} .
\end{cases}
\end{gather}

These functions are closely related, but not identical, to Zagier's single-valued version of the classical polylogarithms from~\eqref{eq:P_def}. The relationship is most conveniently expressed in terms of the function
\begin{gather*}
\cP_n(x) := \sum_{k=0}^{n-1}\frac{B_k}{k!} \log^k|x|^2 \big(\li_{n-k}(x) -(-1)^n\li_{n-k}(\overline{x})\big)
= \begin{cases}
2 P_n(x)&\text{if $n$ odd} ,
\\
2{\rm i} P_n(x) &\text{if $n$ even} .
\end{cases}
\end{gather*}

\begin{Proposition}\label{prop:svLi}
For $n>0$ and $x\in\mathbb{C}\setminus\{0\}$, we have
\begin{gather*}
\svmot\big(\li_n^{\mot}(x)\big) = \sum_{k=0}^{n-1}\frac{\log^{k}|x|^2}{(k+1)!} \cP_{n-k}(x) .
 \end{gather*}
\end{Proposition}

\begin{proof}We directly compute the right hand side, and show that it gives the expression for $ \svmot(\li^\mot_n(x))$ from~(\ref{eq:svLi}). We have
\begin{gather*}
 \sum_{k=0}^{n-1}\frac{\log^{k}|x|^2}{(k+1)!} \cP_{n-k}(x) =
 \sum_{k=0}^{n-1} \sum_{\ell=0}^{n-k-1} \frac{B_\ell}{\ell!} \frac{\log^{k+\ell}{|x|^2}}{(k+1)!} \big( \li_{n-k-\ell}(x) - (-1)^{n-k} \li_{n-k-\ell}(\overline{x}) \big) .
\end{gather*}
By reindexing the sum with $\alpha = k + \ell$, we find it is equal to
\begin{gather}
{} = \sum_{\alpha=0}^{n-1} \sum_{\ell=0}^\alpha \frac{B_\ell}{\ell!} \frac{\log^\alpha{|x|^2}}{(\alpha + 1 - l)!} \big( \li_{n - \alpha}(x) - (-1)^{n+\alpha-l} \li_{n-\alpha}(\overline{x}) \big) \nonumber
 \\
\begin{split}
&{} = \sum_{\alpha=0}^{n-1} \Bigg( \sum_{\ell=0}^\alpha \frac{B_\ell}{\ell!} \frac{1}{(\alpha + 1 - \ell)!} \Bigg) \li_{n - \alpha}(x) \log^\alpha{|x|^2}
 \\
 & \hphantom{=}
{} - (-1)^n \sum_{\alpha=0}^{n-1} \Bigg( \sum_{\ell=0}^\alpha (-1)^{\ell} \frac{B_\ell}{\ell!} \frac{1}{(\alpha + 1 - \ell)!} \Bigg) (-\log{|x|^2})^\alpha \li_{n-\alpha}(\overline{x}) .\end{split} \label{eqn:svli_to_p_sum}
\end{gather}
We notice
\begin{align*}
 \sum_{\ell=0}^\alpha \frac{B_\ell}{\ell!} \frac{1}{(\alpha + 1 - \ell)!}
 & = \frac{1}{(\alpha+1)!} ( B_{\alpha+1}(1) - B_{\alpha+1} ) \\
 & = \frac{1}{(\alpha+1)!} ( (-1)^{\alpha+1} B_{\alpha+1} - B_{\alpha+1} ) \\
 & = \begin{cases}
 1, & \alpha = 0, \\
 0, & \alpha \neq 0 ,
 \end{cases}
\end{align*}
where we have used the symmetry $B_{\alpha+1}(1-x) = (-1)^{\alpha+1} B_{\alpha+1}(x)$, to find that $B_{\alpha+1}(1) = (-1)^{\alpha+1} B_{\alpha+1}$. The second case above follows by combining the odd $\alpha > 0$ case where the terms cancel, and the even $\alpha > 0$ case where the terms are identically zero. Likewise
\begin{gather*}
 \sum_{\ell=0}^\alpha \frac{B_\ell}{\ell!} \frac{(-1)^\ell}{(\alpha + 1 - \ell)!}
 = \frac{(-1)^{\alpha+1}}{(\alpha+1)!} (B_{\alpha+1}(-1) - B_{\alpha+1}) .
\end{gather*}
Now using the symmetry and multiplication theorems for Bernoulli polynomials, we have
\begin{gather*}
B_{\alpha+1}(-x) = (-1)^{\alpha+1} B_{\alpha+1}(x) + (\alpha+1) (-1)^{\alpha+1} x^{\alpha+1-1},
\end{gather*}
so
\begin{gather*}
B_{\alpha+1}(-1) = (-1)^{\alpha+1} B_{\alpha+1}(1) + (\alpha+1) (-1)^{\alpha+1} = B_{\alpha+1} + (\alpha+1) (-1)^{\alpha+1}.
\end{gather*}
So this sum is equal to
\begin{gather*}
 \frac{(-1)^{\alpha+1}}{(\alpha+1)!} \big( B_{\alpha+1} + (\alpha+1) (-1)^{\alpha+1} - B_{\alpha+1} \big) = \frac{(\alpha+1)}{(\alpha+1)!} = \frac{1}{\alpha!} .
\end{gather*}
Inserting these evaluations into \eqref{eqn:svli_to_p_sum} shows that it is equal to
\begin{gather*}
 = \li_{n}(x) - (-1)^n \sum_{\alpha=0}^{n-1} \frac{(-\log{|x|^2})^\alpha}{\alpha!} \li_{n-\alpha}(\overline{x}) = \svmot(\li_n^\mot(x)) ,
\end{gather*}
as claimed.
\end{proof}

\begin{Remark}
An alternative construction of single-valued analogues of MPL's was presented in \cite{Zhao}. Neither of the single-valued versions from \cite{brownnotesmot} or \cite{Zhao} satisfy exclusively clean functional equations. For \cite{brownnotesmot} this follows from the functoriality of the construction, for example: applying the single-valued map to the functional equation
\begin{gather*}
 \li^\mot_2(x) + \li^\mot_2(1-x) = - \frac{1}{2} \big(\log^\mot(-x) \big)^2 + \zeta^\mot(2)
\end{gather*}
produces the following identity between single-valued functions
\begin{gather*}
 \sv^\mot\big( \li^\mot_2(x)\big) + \sv^{\mot}\big(\li^\mot_2(1-x)\big) = - \frac{1}{2} \big(\log^\mot|x|^2 \big)^2 ,
\end{gather*} which still retains a product term on the right hand side. For~\cite{Zhao}, see the explicit example Section~2.9.3 in loc. cit.
\end{Remark}
\end{Example}

\section{Clean single-valued polylogarithms}\label{sec:clean}

Throughout this section (and the following) all MPL's with non-generic arguments are understood to be regularised by introducing suitable tangential base-points, cf.\ the comment in Section~\ref{sec:definitions} and~\cite[Chapter~15]{Deligne:1989}.

\subsection{Definition}%\label{sec:clean}
We can apply the construction of the Dynkin operator from Section~\ref{sec:Dynkin} to the commutative graded connected Hopf algebra $\cPdrpl$. We can compose the projector $\Pi$ with the projection~$\pi^{\dr}$, the single-valued projection $\sv$ and the period map to obtain an algebra morphism \mbox{$\cR \colon\cPmpl\to \mathbb{C}$}:
\begin{gather*}
\cR := \per\circ \sv\circ \Pi\circ \pi^{\dr} ,
\end{gather*}
where $\Pi = Y^{-1}D$ acts as defined in Section~\ref{sec:Dynkin}.

\begin{Definition}
\label{def:Icsv}
The \emph{clean single-valued multiple polylogarithms $\Icsv$} are defined by
\begin{gather*}
\Icsv(x_0;x_1,\ldots,x_n;x_{n+1}) := \mathfrak{R}_n\big[\cR(\Imot(x_0;x_1,\ldots,x_n;x_{n+1}))\big] ,
\end{gather*}
where $\mathfrak{R}_n$ is defined in~\eqref{eq:Rn_def}.
\end{Definition}

Theorem~\ref{thm:main} follows immediately, from the definition of the clean single-valued {multiple} polylogarithms and the properties of $\cR$. Indeed, since the latter lie in the image of $\per\circ \sv$, they are both real-analytic and single-valued functions. Moreover, let
\begin{gather*}
A := \sum_{k=1}^Kc_k \Imot(x_{k,0};x_{k,1},\ldots,x_{k,n};x_{k,n+1})\in \cPm_{\mathrm{MPL},>0}\cdot\cPm_{\mathrm{MPL},>0} .
\end{gather*}
Proposition~\ref{prop:ker} implies that $\cPm_{\mathrm{MPL},>0}\cdot\cPm_{\mathrm{MPL},>0}\subseteq \ker \cR$, and so
\begin{gather*}
0 = \mathfrak{R}_n\left[\cR(A)\right] = \sum_{k=1}^Kc_k \Icsv(x_{k,0};x_{k,1},\ldots,x_{k,n};x_{k,n+1}) ,
\end{gather*}
where $\mathfrak{R}_n$ was defined above after \eqref{eq:P_def}.

\begin{Remark}
It is possible to use Theorem~\ref{thm:main} to obtain identities among (non-clean) single-valued polylogarithms. Indeed, it is often easier to find identities modulo product terms, e.g., by starting from identities that hold modulo shuffle products at the symbol level (cf.~\cite{CharltonPhD,Charlton:2019gvp}). The combinatorics involved in $\cR$ will restore all the product terms necessary to obtain a numerical identity between single-valued polylogarithms, up to a single constant of integration. In~some cases this may even give hints for valid identities among the non-single-valued analogues, e.g., by dropping all terms depending on the complex-conjugated variables, and accounting for factors of 2 introduced by the single-valued map on real constants (e.g.,~\eqref{eq:sv_zeta}). Conversely, the combinatorics involved in $\cR$ can be applied directly to the symbol Hopf algebra to restore the (functional) product terms in a modulo products identity between functions (of holomorphic variables) at the symbol level; one can then study product terms involving constants iteratively via slices of the coaction.
\end{Remark}

\begin{Remark}The restriction to the real (resp.\ imaginary) part for odd (resp.\ even) weights in Definition~\ref{def:Icsv} can be motivated by the fact that the other parity can be expressed entirely in terms of products of lower weights functions. To see this, we start from the following property of the single-valued projection on de Rham MPL's (cf., e.g.,~\cite{DelDuca:2016lad}):
\end{Remark}
\begin{Proposition}\label{prop:Fsv}
For every $x\in\cPdr_{\mathrm{MPL},n}$, we have
\begin{gather*}
F_{\infty}\sv(x) = (-1)^n\sv S(x) .
\end{gather*}
\end{Proposition}
\begin{proof}
The following {two} properties are well known and hold in any commutative, graded and connected Hopf algebra {(see, e.g., \cite[Proposition~I.7.1]{Manchon:2001bf} and references therein)}:
\begin{gather*}
(S\otimes S)\tau\Delta =\Delta S ,
\\
S^2 = \id .
\end{gather*}
Since $\Deltatilde$ and $\Stilde = (-1)^Y\Sigma$ are defined by the same combinatorial formulas as $\Delta$ and $S$, but with~$\Idr$ replaced by $\Imot$, it is easy to see that the following identities hold:
\begin{gather*}
(\Sigma\otimes\Sigma)\tau \Deltatilde = \Deltatilde S (-1)^Y,
\\
\Sigma^2 = \id .
\end{gather*}
This gives, with $F_{\infty}^2=\id$,
\begin{align*}
F_{\infty}\sv & = F_{\infty}m(F_{\infty}\Sigma\otimes \id)\Deltatilde
=m(\Sigma\otimes F_{\infty})\Deltatilde\\
& =m(\id\otimes F_{\infty}\Sigma)\tau\Deltatilde S(-1)^Y
=m\tau(F_{\infty}\Sigma\otimes \id)\Deltatilde S(-1)^Y\\
& =\sv S(-1)^Y .
\tag*{\qed}\end{align*}
 \renewcommand{\qed}{}
\end{proof}

\begin{Corollary}
Let $x\in \cPdr_{\mathrm{MPL},>0}$. Then
\begin{gather*}
\sv(x) +(-1)^YF_{\infty}\sv(x)\in \cPm_{\mathrm{MPL},>0}\cdot \cPm_{\mathrm{MPL},>0} .
\end{gather*}
\end{Corollary}
\begin{proof}
Let $x\in \cPdr_{\mathrm{MPL},n}$, $n>0$.
Proposition~\ref{prop:Fsv} implies
\begin{gather*}\begin{split}
\sv(x) +(-1)^nF_{\infty}\sv(x) & = \sv\left(x+S(x)\right) \\
& = -\sv m(S\otimes \id)\Delta'(x)\in \cPm_{\mathrm{MPL},>0}\cdot \cPm_{\mathrm{MPL},>0} ,
\end{split}\end{gather*}
where the last equality follows from $m(S\otimes \id)\Delta(x) = 0$.
\end{proof}

\subsection{Elementary properties of clean single-valued polylogarithms}
\subsubsection{Shuffle products, path composition and reversal}
The clean single-valued polylogarithms inherit the basic properties of iterated integrals (see Section~\ref{sec:definitions}). Using Theorem~\ref{thm:main}, we see that they take the form:
\begin{enumerate}\itemsep=0pt
\item \emph{Path reversal: }
\begin{gather}\nonumber
\Icsv(x_{n+1};x_n,\ldots,x_1;x_0) = (-1)^n \Icsv(x_0;x_1,\ldots,x_n;x_{n+1}) .
\end{gather}
\item \emph{Path composition:}
\begin{gather}\nonumber
\Icsv(x_0;x_1,\ldots,x_n;x_{n+1}) = \Icsv(x_0;x_1,\ldots,x_n;x) + \Icsv(x;x_{1},\ldots,x_n;x_{n+1}) .
\end{gather}
\item \emph{Shuffle product:} \label{item:shuffle}
\begin{gather}\begin{split}\nonumber
\sum_{\sigma\in\Sigma(m,n)}\!\!\!\Icsv(x_0;x_{\sigma(1)},\ldots,x_{\sigma(m+n)};x) = 0 .
\end{split}\end{gather}
\end{enumerate}

\subsubsection{Reversal of arguments}
\begin{Proposition}\label{prop:reversal}
For $n>0$, we have
\begin{gather}\nonumber
\Icsv(x_{0};x_n,\ldots,x_1;x_{n+1}) = (-1)^{n+1} \Icsv(x_0;x_1,\ldots,x_n;x_{n+1}) .
\end{gather}
\end{Proposition}

\begin{proof}
Consider the shuffle algebra generated by the letters $x_1,\ldots,x_n$. It is a Hopf algebra whose coproduct is deconcatenation and the antipode is the reversal of words, up to a sign:
\begin{gather*}
\Delta_{\mathrm{sh}}(w) = \sum_{uv=w}u\otimes v ,
\\
S_{\mathrm{sh}}(w) = (-1)^{|w|} \widetilde{w} ,
\end{gather*}
where $\widetilde{w}$ is the word $w$ in reverse order, and $|w|$ its length. If $m_{\mathrm{sh}}$ denotes the shuffle multiplication, we have
\begin{gather*}
0 = m_{\mathrm{sh}}(S_{\mathrm{sh}}\otimes\id)\Delta_{\mathrm{sh}}(w) = w + (-1)^{|w|} \widetilde{w} + m_{\mathrm{sh}}(S_{\mathrm{sh}}\otimes\id)\Delta'_{\mathrm{sh}}(w) .
\end{gather*}
If we take $w=x_1\cdots x_n$, we see that $x_1\cdots x_n+(-1)^nx_n\cdots x_1$ must vanish modulo non-trivial products. This relations must hold in every shuffle algebra, and so in particular the combination $\Imot(x_{0};x_n,\ldots,x_1;x_{n+1}) + (-1)^n \Imot(x_0;x_1,\ldots,x_n;x_{n+1})$ must vanish modulo non-trivial products, from which we deduce Proposition~\ref{prop:reversal} via Theorem~\ref{thm:main}.
\end{proof}

\subsubsection{Unshuffling of leading zeros}
\begin{Proposition}\label{prop:unshuffling}
For any $k \in \mathbb{Z}_{\geq0}$ the following holds
\begin{gather*}
\Icsv\big(0;\{0\}^{k},x_1,\{0\}^{n_1-1},\ldots,x_r,\{0\}^{n_r-1};x_{r+1}\big)
\\ \qquad
{}=(-1)^k \sum_{i_1+\cdots+i_r=k}
\binom{n_1+i_1-1}{i_1}\cdots \binom{n_r+i_r-1}{i_r}
\\ \qquad\hphantom{=}
{}\times \Icsv\big(0;x_1,\{0\}^{n_1-1+i_1},\ldots,x_r,\{0\}^{n_r-1+i_r};x_{r+1}\big) .
\end{gather*}
\end{Proposition}

\begin{proof}
This is proven by induction. The case $k = 0$ holds trivially wherein both sides are identical, so we may suppose this formula holds for all $n \leq k$. Now observe
\begin{gather*}
\Icsv\big(0; \{0\}^{k+1}, x_1, \{0\}^{n_1-1}, \ldots, x_r, \{0\}^{n_r-1}; x_{r+1}\big)
\\ \qquad
{}= - \frac{1}{k+1} \sum_{j = 1}^{r} n_j \Icsv\big(0; \{0\}^{k}, x_1, \{0\}^{n_1-1}, \ldots, x_j, \{0\}^{(n_j-1) + 1}, \ldots, x_r, \{0\}^{n_r-1}; x_{r+1}\big) ,
\end{gather*}
using the shuffle product property (\ref{item:shuffle}) above. Substituting the induction assumption into the second line shows that each term in the result is indexed by a composition $i_1' + \cdots + i_r' = k+1$, added to the exponents $n_1-1, \ldots, n_r-1$ of the original integral. Therefore we need only to compute the coefficient and check that it matches the one claimed in the formula.

This coefficient is
\begin{gather*}
 -\frac{1}{k+1} \sum_{j=1}^r n_j (-1)^k \binom{n_1 + i_1' - 1}{i_1'} \cdots \binom{(n_j + 1) + (i_j' - 1) - 1}{(i_j' - 1)} \cdots \binom{n_r + i_r' - 1}{i_r'} ,
\end{gather*}
where $i_1' + \cdots + i_r' = k+1$. Observe that
\begin{gather*}
 n_j \binom{(n_j + 1) + (i_j' - 1) - 1}{(i_j' - 1)} = i_j' \binom{n_j + i_j' - 1}{i_j'} ,
\end{gather*}
so the coefficient is equal to
\begin{gather*}
\frac{ (-1)^{k+1} }{k+1} \sum_{j=1}^r i_j' \binom{n_1 + i_1' - 1}{i_1'} \cdots \binom{n_j + i_j' - 1}{i_j'} \cdots \binom{n_r + i_r' - 1}{i_r'}
\\ \qquad
{}= \frac{ (-1)^{k+1} }{k+1} \underbrace{\Bigg( \sum_{j=1}^r i_j' \Bigg)}_{k+1} \binom{n_1 + i_1' - 1}{i_1'} \cdots \binom{n_r + i_r' - 1}{i_r'}
\\ \qquad
{}= (-1)^{k+1} \binom{n_1 + i_1' - 1}{i_1'} \cdots \binom{n_r + i_r' - 1}{i_r'} ,
\end{gather*}
as claimed.
\end{proof}

\subsection[Recursion and the total holomorphic differential of \protect{I\textasciicircum{}\{csv\}}]
{Recursion and the total holomorphic differential of $\boldsymbol {I^{\rm csv}}$}

\begin{Proposition}\label{prop:recursion}
Write the following shorthand \begin{gather}
C(x_0;x_{1},\ldots,x_{n};x_{n+1}) := \cR\big[\Idr(x_0;x_{1},\ldots,x_{n};x_{n+1})\big] ,
\end{gather}
then $C$ satisfies the following recursive formula
\begin{gather}
\begin{split}
& C(x_0;x_1,\ldots,x_n;x_{n+1}) = \Isv(x_0;x_1,\ldots,x_n;x_{n+1})\\
& {}- \frac{1}{n} \Bigg[ 	\sum_{\substack{0\le i< j\le n \\ (i,j) \neq (0,n)}}(j-i)
 \Isv(x_0;x_1,\ldots,x_i,x_{j+1},\ldots,x_{n};x_{n+1})
 C(x_i;x_{i+1},\ldots,x_{j};x_{j+1})\Bigg] .
\end{split} \label{eq:recursion}
\end{gather}
\end{Proposition}
	
\begin{proof}
Since $Y=\id\star D$, the Dynkin operator satisfies the following recursion, valid in every graded commutative Hopf algebra $H$:
\begin{align}
D(x) &= n x - m(\id\otimes D)\Delta'(x) \nonumber
\\
 &= n x - m(\id\otimes (Y\cdot \Pi))\Delta'(x) ,\qquad x \in H_n ,\quad n>0 .
\label{eqn:Drecursion}
\end{align}
This gives
\begin{align*}
(\sv\circ \Pi)(\Idr(x_0;x_1,\ldots,x_n;x_{n+1})) ={}& \Isv(x_0;x_1,\ldots,x_n;x_{n+1})
\\
& -\frac{1}{n} m(\sv\otimes (Y\cdot (\sv\circ\Pi)))\Delta'(\Idr(x_0;x_1,\ldots,x_n;x_{n+1})) .
\end{align*}
Because of the projector $\Pi$ in the second entry of the tensor product, we only need to consider terms in the reduced coproduct that have no product in the second entry, cf.~\eqref{eq:coaction}). This constraint is described via the infinitesimal coproduct (cf.~\cite{Brown:2011ik}), and one obtains the recursive formula directly therefrom, wherein we must exclude the case $(i,j) = (0,n)$ because we have taken the reduced coproduct.
	\end{proof}

Since $\Icsv(x_0;x_{1},\ldots,x_{n};x_{n+1}) = \mathfrak{R}_n\big[C(x_0;x_{1},\ldots,x_{n};x_{n+1})\big]$, \eqref{eq:recursion} can be interpreted as a~recursion for the clean single-valued MPL's.

We recall now that the total differential of an MPL is given by
\begin{gather*}
\mathrm{d}I(x_0; x_1,\ldots, x_n; x_{n+1}) = \sum_{k=1}^n I\big(x_0; x_1,\ldots, \widehat{x_k} ,\ldots, x_n; x_{n+1}\big) \mathrm{d} I(x_{k-1}; x_k; x_{k+1}) .
\end{gather*}
The function $C$ satisfies a similar formula for the total holomorphic differential $\partial$. There is no correspondingly simple formula for the total antiholomorphic differential, though, since the single-valued map only preserves the holomorphic differential.

\begin{Proposition}
The total holomorphic differential of the function $C$ is given in weight $1$ by
\begin{gather*}
\partial C(x_0; x_1; x_2) = \mathrm{d} I(x_0; x_1; x_2) ,
\end{gather*}
and in weight $n > 1$ by
\begin{gather*}
\partial C(x_0; x_1,\ldots,x_n; x_{n+1}) = \frac{n-1}{n} \Bigg[
\sum_{k=1}^n C(x_0; x_1,\ldots,\widehat {x_k},\ldots,x_n;x_{n+1}) \mathrm{d} I(x_{k-1}; x_k; x_{k+1}) \\ \hphantom{\partial C(x_0; x_1,\ldots,x_n; x_{n+1}) = \frac{n-1}{n} \Bigg[}
{} - C(x_0; x_1,\ldots, x_{n-1}; x_n) \mathrm{d}I(x_0; x_n; x_{n+1})
\\ \hphantom{\partial C(x_0; x_1,\ldots,x_n; x_{n+1}) = \frac{n-1}{n} \Bigg[}
{} - C(x_1; x_2,\ldots,x_n; x_{n+1}) \mathrm{d} I(x_0; x_1,x_{n+1})\Bigg].
\end{gather*}
	
\begin{proof}
We prove this via the recursion in Proposition \ref{prop:recursion}; one can check directly the case $n = 1$. Namely, we aim to compute
\begin{gather*}
\partial C(x_0; x_1; x_2) = \partial \Isv(x_0; x_1; x_2) ,
\end{gather*}
where we have computed $C(x_0; x_1;x_2) = \Isv(x_0; x_1; x_2)$ either via the recursion in Proposition~\ref{prop:recursion} or directly from the definition. The holomorphic derivative is reserved by the single-valued map, so we immediately obtain
\begin{gather*}
\partial C(x_0; x_1; x_2) = \sv \mathrm{d} I(x_0; x_1; x_2) = \mathrm{d} I(x_0; x_1; x_2) ,
\end{gather*}
since the total derivative of the weight 1 function is rational, and hence single-valued already.
		
Note that the total holomorphic differential of $\Icsv$ is given by the same formula as for the total differential of $I$, with $I \mapsto \Isv$ but with $\mathrm{d} I$ unchanged, namely
\begin{gather*}
\partial \Isv(x_0; x_1,\ldots, x_n; x_{n+1}) =
\sum_{k=1}^n \Isv\big(x_0; x_1,\ldots, \widehat{x_k} ,\ldots, x_n; x_{n+1}\big) \mathrm{d} I(x_{k-1}; x_k; x_{k+1}) .	
\end{gather*}
Now for weight $n$ the recursion implies
\begin{gather*}
\partial C(x_0; x_1,\ldots,x_n; x_{n+1}) =
\partial \Isv(x_0; x_1,\ldots,x_n; x_{n+1})
\\ \qquad
{} -\sum_{\substack{0\leq i<j\leq n \\ (i,j) \neq (0,n)}} \bigg[ \frac{j-i}{n} \cdot \partial \Isv(x_0; x_1, \ldots, x_i, x_{j+1}, \ldots, x_{n}; x_{n+1})
 C(x_i; x_{i+1}, \ldots, x_{j}; x_{j+1})
\\ \qquad\hphantom{-\sum_{\substack{0\leq i<j\leq n \\ (i,j) \neq (0,n)}} \bigg[}
{} + \frac{j\!-\!i}{n} \cdot \Isv(x_0; x_1, \ldots, x_i, x_{j+1}, \ldots, x_{n}; x_{n+1})
 \partial C(x_i; x_{i+1}, \ldots, x_j; x_{j+1}) \bigg] .
\end{gather*}
Then by taking care of terms which cross the jump $x_i$, $x_{j+1}$, we can write
\begin{gather*}
 \partial \Isv(x_0; x_1, \ldots, x_i, x_{j+1}, \ldots, x_{n}; x_{n+1})
\\ \qquad
{}= \sum_{k=1}^{i-1} \Isv\big(x_0; x_1, \ldots, \widehat{x_k}, \ldots, x_i, x_{j+1}, \ldots, x_n; x_{n+1}\big) \mathrm{d}I (x_{k-1}; x_k; x_{k+1})
\\ \qquad\hphantom{=}
{} + \sum_{k=j+2}^n \Isv\big(x_0; x_1, \ldots, x_i, x_{j+1}, \ldots, \widehat{x_k}, \ldots, x_n; x_{n+1}\big) \mathrm{d}I (x_{k-1}; x_k; x_{k+1})
\\ \qquad\hphantom{=}
{}+ \Isv(x_0; x_1,\ldots,x_{i-1}, x_{j+1}, \ldots, x_n; x_{n+1}) \mathrm{d}I(x_{i-1}; x_i; x_{j+1})
\\ \qquad\hphantom{=}
{} + \Isv(x_0; x_1,\ldots,x_i, x_{j+2}, \ldots, x_n; x_{n+1}) \mathrm{d}I(x_{i}; x_{j+1}; x_{j+2}) .
\end{gather*}
Correspondingly, by the induction hypothesis, for $j-i > 1$ we get
\begin{gather*}
\partial C(x_i; x_{i+1}, \ldots, x_j; x_{j+1})
\\ \qquad
{}= \frac{j-i - 1}{j-i} \Bigg[
\sum_{k=i+1}^j C\big(x_i; x_{i+1},\ldots,\widehat {x_k},\ldots,x_j;x_{j+1}\big) \mathrm{d} I(x_{k-1}; x_k; x_{k+1})
\\ \qquad\hphantom{=\frac{j-i - 1}{j-i} \Bigg[}
{} - C(x_i; x_{i+1},\ldots, x_{j-1}; x_j) \mathrm{d}I(x_i; x_j; x_{j+1})
\\ \qquad\hphantom{=\frac{j-i - 1}{j-i} \Bigg[}
{} - C(x_{i+1}; x_{i+2},\ldots,x_j; x_{j+1}) \mathrm{d} I(x_i; x_{i+1}; x_{j})\Bigg] ,
\end{gather*}
and for $j-i = 1$ we find $\partial C(x_i; x_{i+1}, \ldots, x_j; x_{j+1}) = \mathrm{d} I(x_i; x_{i+1}; x_{i+2})$. \medskip
		
		Now we note that the $\mathrm{d}I$ terms in $\partial C$ occur only with certain fixed patterns, namely\linebreak$\mathrm{d}I(x_{k-1};x_{k};x_{k+1})$ wherein all arguments are consecutive, and either $\mathrm{d}I(x_{i}; x_{i+1}; x_j)$ and/or $\mathrm{d}I(x_i; x_j; x_{j+1})$ wherein the first two, respectively last two, are consecutive arguments.
		
		So first we ask what the coefficient of $\mathrm{d}I(x_{k-1}; x_k; x_{k+1})$, for fixed $k$, is. It is seen to be the following, where the first line arises from differentiating the $\Isv$ appearing outside the sum, the second and third line arise from differentiating the $\Isv$ inside the sum, the fourth line from differentiating $C$ inside the sum, and the last line deals with the edge case where one has differentiated $C(x_i; x_{i+1}; x_{i+1})$ when $j = i+1$, for $j = k$ in order to obtain $\mathrm{d}I(x_{k-1}; x_k; x_{k+1})$. (Note that the corresponding term in line 4 gives 0 in this case, so no extra restriction is necessary there.)
\begin{gather*}
\Isv\big(x_0; x_1,\ldots,\widehat{x_k},\ldots,x_n;x_{n+1}\big)
\\ \qquad
{} -\sum_{\substack{0 \leq i < j \leq n \\ (i,j) \neq (0,n) \\ k < i}} \frac{j-i}{n} \Isv\big(x_0, x_1,\ldots,\widehat{x_k},\ldots,x_i,x_{j+1},\ldots,x_n;x_{n+1}\big) C(x_i; x_{i+1}, \ldots, x_j; x_{j+1})
\\ \qquad
{} - \sum_{\substack{0 \leq i < j \leq n \\ (i,j) \neq (0,n) \\ j+1 < k}}\!\!\! \frac{j-i}{n} \Isv\big(x_0, x_1,\ldots,,x_i,x_{j+1},\ldots,\widehat{x_k},\ldots,x_n;x_{n+1}\big) C(x_i; x_{i+1}, \ldots, x_j; x_{j+1})
\\ \qquad
{} - \sum_{\substack{0 \leq i < j \leq n \\ (i,j) \neq (0,n) \\ i < k < j+1}} \frac{j-i}{n} \, \frac{ j-i-1}{j-i} \Isv(x_0, x_1,\ldots,,x_i,x_{j+1},\ldots,x_n;x_{n+1})
\\ \qquad \hphantom{- }
{}\times C\big(x_i; x_{i+1}, \ldots,\widehat{x_k},\ldots, x_j; x_{j+1}\big)
 - \frac{1}{n} \Isv(x_0, x_1,\ldots,x_{k-1},x_{k+1},\ldots,x_n;x_{n+1}) .
\end{gather*}
This is nothing but the recursion for $C$ applied to
\begin{gather*}
\frac{n-1}{n} C\big(x_0; x_1,\ldots,\widehat{x_k},\ldots,x_n;x_{n+1}\big) .
\end{gather*}
		
Now we ask about the coefficient of $\mathrm{d}I(x_k, x_{k+1}; x_{\ell+1})$, where $k + 2 < \ell+1$ for non-consecutive arguments. This term arises from either differentiating the $\Isv$ term in
\begin{gather*}
\Isv(x_0; x_1,\ldots,x_{k},x_{k+1}, x_{\ell+1}, \ldots, x_n; x_{n+1}) C(x_{k+1}; x_{k+2}, \ldots, x_\ell; x_{\ell+1}) ,
\end{gather*}
where $(i,j) = (k+1,\ell)$, or by differentiating the $C$ term in
\begin{gather*}
\Isv(x_0; x_1,\ldots,x_{k-1},x_k, x_{\ell+1}, \ldots, x_n; x_{n+1}) C(x_k; x_{k+1}, \ldots, x_\ell; x_{\ell+1}) ,
\end{gather*}
where $(i,j) = (k,\ell)$. We note that for each choice of $(k,\ell) \neq (0,n)$ such that $0 \leq k < \ell \leq n$ with $k + 1 < \ell$, both terms contribute to the coefficient. However when $(k,\ell) = (0,n)$ only the~$\Isv$ derivative contributes as $(i,j) = (k,\ell) = (0,n)$ is excluded from the summation, and the~$C$ derivative therewith.
		
When $(k,\ell) \neq (0,n)$ we find the coefficient of $\mathrm{d}I(x_{k};x_{k+1};x_\ell)$ to be
\begin{gather*}
\frac{\ell - (k+1)}{n} \Isv\big(x_0; x_1,\ldots,x_{k},\widehat{x_{k+1}}, x_{\ell+1}, \ldots, x_n; x_{n+1}\big) C(x_{k+1}; x_{k+2}, \ldots, x_\ell; x_{\ell+1})
\\ \qquad
{}- \frac{\ell \!-k}{n} \frac{\ell-k\!-1}{\ell-k} \Isv(x_0; x_1,\ldots,x_{k-1},x_k, x_{\ell+1}, \ldots, x_n; x_{n+1})
\\ \qquad \hphantom{=}
{}\times C\big(\widehat{x_k}; x_{k+1}, \ldots, x_\ell; x_{\ell+1}\big)\! = 0.
\end{gather*}
However when $(k,\ell) = (0,n)$ we find the coefficient of $\mathrm{d}I(x_0;x_1;x_{n+1})$ to be
\begin{gather*}
\frac{\ell - (k+1)}{n} \Isv\big(x_0;\widehat{x_1};x_n\big) C(x_1; x_2,\ldots,x_n; x_{n+1}) = \frac{n-1}{n} C(x_1; x_2,\ldots,x_n; x_{n+1}) .
\end{gather*}
		Exactly the same consideration applies to the coefficient of $\mathrm{d}I(x_k; x_{\ell};x_{\ell+1})$, wherein the terms pairwise cancel, except for the case $(k,\ell) = (0,n)$, where only one term occurs which cannot cancel. This completes the proof of the formula for the total holomorphic derivative of $C$.
 	\end{proof}
\end{Proposition}

We note that this differential formula is closely related to a recursion for the mod-products symbol $ \Pi_\bullet \mathcal{S}$ of an iterated integral, as given in \cite[equation (4)]{Charlton:2019gvp}
\begin{align*}
\Pi \mathcal{S} \big( I^{\dr}(x_0; \ldots; x_{n+1}) \big)
={}& \sum_{j=1}^n
\Pi \mathcal{S} \big( I^{\dr}(x_0; x_1, \ldots, \widehat{x_j}, \ldots, x_n; x_{n+1}) \big) \otimes I^{\dr}(x_{j-1}; x_j; x_{j+1})
\\
&- \Pi \mathcal{S} \big( I^{\dr}(x_1; x_2, \ldots, x_{n}; x_{n+1}) \big) \otimes I^{\dr}(x_0; x_1; x_{n+1})
\\
&- \Pi \mathcal{S} \big( I^{\dr}(x_0; x_1, \ldots, x_{n-1}; x_{n}) \big) \otimes I^{\dr}(x_0; x_n; x_{n+1}) .
\end{align*}

\section{Examples in small depths}\label{sec:examples}

In this section we present results for clean single-valued polylogarithms in small depths.
The path composition formula, together with Proposition~\ref{prop:unshuffling} and
\begin{gather*}
\Icsv(k\cdot x_0;k\cdot x_1,\ldots,k\cdot x_n;k\cdot x_{n+1}) =\Icsv( x_0; x_1,\ldots, x_n; x_{n+1}) ,
\end{gather*}
with $x_1\neq x_0$, $x_n\neq x_{n+1}$ and {$k\in\mathbb{C}\setminus\{0\}$}
(this identity follows immediately from the corresponding identity for MPL's, where it is a direct consequence of the integral representation~\eqref{eq:G_def}), imply that it is sufficient to consider the functions
\begin{gather*}
\Icsv_{m_1,\ldots,m_k}(x_1,\ldots,x_k) := \Icsv\big(0;x_1,\{0\}^{m_1-1},\ldots,x_{k},\{0\}^{m_k-1};1\big) .
\end{gather*}
We will also use the objects
\begin{gather*}
I^{\bullet}_{m_1,\ldots,m_k}(x_1,\ldots,x_k) :=
I^{\bullet}\big(0;x_1,\{0\}^{m_1-1},\ldots,x_{k},\{0\}^{m_k-1};1\big),\qquad \bullet\in\{\mot,\dr\},
\\
\Isv_{m_1,\ldots,m_k}(x_1,\ldots,x_k) := \svmot\big(\Imot(0;x_1,\{0\}^{m_1-1},\ldots,x_{k},\{0\}^{m_k-1};1)\big).
\end{gather*}

%%%%%%%%%%%%%%%%
\subsection{Results in depth 1}

\begin{Proposition}\label{prop:Icsv_d1}
\begin{gather*}
\Icsv_n(x) = \mathfrak{R}_n\bigg[\Isv_n(x) + \frac{1}{n} \log|x|^2 \Isv_{n-1}(x)\bigg] ,
\end{gather*}
where we interpret $\Isv_n(x)=0$ when $n\le 0$.
\end{Proposition}

\begin{proof}
It follows from~\eqref{eq:id_S_D} that if $x\in \cPdrpl$ has weight $n$, we obtain the recursion
\begin{gather*}
D(x) = n x - m(\id\otimes D)\Delta'(x) .
\end{gather*}
Proposition~\ref{eq:Delta_Li} then implies
\begin{align}
\begin{split}
\Pi\big(\Idr_n(x)\big) & = -\Pi\bigg(\li^{\dr}_n\bigg(\frac1x\bigg)\bigg)
\\
& =-\li^{\dr}_n\bigg(\frac1x\bigg) + \frac{1}{n}\sum_{k=1}^{n-1}\frac{1}{k!} \li^{\dr}_{n-k}\bigg(\frac1x\bigg) D\bigg(\log^{\dr}\bigg(\frac1x\bigg)^k\bigg)
\\
& =\Idr_n(x)+\frac{1}{n} \log^{\dr}x \Idr_{n-1}(x) .
\end{split}\label{eqn:PiIn}
\end{align}
The claim follows upon acting with $\per\circ \sv$, and taking the real or imaginary part.
\end{proof}

The functions $\Icsv_n(x)$ are real-analytic and single-valued, and they reduce to (single-valued) zeta values for $x=1$ (cf.~\eqref{eq:sv_zeta}),
\begin{gather}
\begin{split}
&\Icsv_{2m}(1) = 0 ,\\
&\Icsv_{2m+1}(1) =-2 \zeta_{2m+1} .
\end{split}\label{eq:clean_zeta_d1}
\end{gather}
This is a special case of the following more general result:
\begin{Corollary}\label{prop:Icsv_root_unity}
Let $\xi_N={\rm e}^{2\pi {\rm i}/N}$, and let $n>1$ and $a$ be integers.
Then
\begin{gather*}
\Icsv_n\big(\xi_N^a\big) = 2 (-1)^n \Cl_n\bigg(\frac{2\pi a}{N}\bigg) ,
\end{gather*}
where $\Cl_n\!(\alpha) := \mathfrak{R}_n\big(\li_n\!({\rm e}^{i\alpha})\big)$ denotes the Clausen function. The same formula also holds for $n=1$ and $a \neq 0 \mod N$.
\end{Corollary}
\begin{proof}
From Propositions~\ref{prop:Icsv_d1} and~\ref{prop:svLi}, we obtain
\begin{align*}
\Icsv_n\big(\xi_N^a\big) & = -\mathfrak{R}_n\big[\svmot\big(\li_n\big(\xi_N^{-a}\big)\big)\big]
\\
& = -\mathfrak{R}_n\big[\li_n\big(\xi_N^{-a}\big) - (-1)^n\li_n\big(\xi_N^{a}\big)\big]
\\
& = 2 (-1)^n \mathfrak{R}_n\big[\li_n\big(\xi_N^{a}\big)\big]
\\
& = 2 (-1)^n \Cl_n\!\bigg(\frac{2\pi a}{N}\bigg) . \tag*{\qed}
\end{align*}
\renewcommand{\qed}{}
\end{proof}

$\Icsv_2(x)$ satisfies a clean version of the five-term relation, and $\Icsv_n(x)$ satisfies for all $n>1$ the inversion relation
\begin{gather*}
\Icsv_n\bigg(\frac1x\bigg) = (-1)^{n+1} \Icsv_n(x) .
\end{gather*}
We see that the functions $\Icsv_n(x)$ have the same properties and satisfy the same identities as the Zagier's single-valued versions of the classical polylogarithms $P_n(x)$ defined in~\eqref{eq:P_def}. However, the two families of functions are not identical, but we have the relation:

\begin{Corollary}\label{cor:Don}
For any $x\in\mathbb{C}\setminus\{0,1\}$ and any $n\in\mathbb{Z}_{>1}$, we have
\begin{gather*}
\Icsv_n(x) =
2 (-1)^n \Bigg[P_n(x) +\frac{1}{n}\sum_{k=1}^{\lceil n/2\rceil-1}(n-2k-1) \frac{\log^{2k}|x|^2}{(2k+1)!} P_{n-2k}(x)\Bigg].
\end{gather*}
Moreover, $\Icsv_1(x)= -2 P_1(x) - \log|x|^2$.
\end{Corollary}

\begin{proof}
It suffices to inject the result of Proposition~\ref{prop:svLi} into the expression for $\Isv_n(x) = -\svmot(\li^{\mot}_n(1/x))$ from Proposition~\ref{prop:Icsv_d1}.
\end{proof}

\begin{Remark}
Corollary~\ref{cor:Don} shows that, for classical polylogarithms (i.e., in depth~1), it is possible to define at least two distinct real-analytic single-valued analogues that satisfy clean functional relations. It would be an interesting question to investigate whether alternative definitions are also possible in higher depth. As a starting point it could be interesting to clarify the relationship between our construction of clean single-valued MPL's and the Lie-period map defined in \cite[Section~2.5]{goncharov2016exponential}. {We are grateful to Cl\'ement Dupont for this remark.}
\end{Remark}

\subsection{Results in depth 2}

\begin{Proposition}
For $(x_1,x_2) \in \mathbb{C}^2\setminus \{(x,y)\colon x\neq 0,1, y\neq 0,1,x\}$, $m_1,m_2\in\mathbb{Z}_{>0}$, and $n=m_1+m_2$, we have
\begin{gather}
\begin{split}
&\Icsv_{m_1,m_2}(x_1,x_2)
= \mathfrak{R}_n\big[\Isv_{m_1,m_2}(x_1,x_2)\big] - \frac{1}{n}\mathfrak{R}_n\Bigg\{m_1 \Isv_{m_1}\bigg(\frac{x_1}{x_2}\bigg) \Isv_{m_2}(x_2)
\\ &\quad\hphantom{=}
 {}- \log|x_2|^2 \Isv_{m_1,m_2-1}(x_1,x_2) + \log\bigg|\frac{x_2}{x_1}\bigg|^2 \bigg[\Isv_{m_1-1,m_2}(x_1,x_2)-\Isv_{m_1-1}\bigg(\frac{x_1}{x_2}\bigg) \Isv_{m_2}(x_2)\bigg]
 \\ &\quad\hphantom{=}
{} +\sum_{r=1}^{m_2}(-1)^{m_1}\binom{n-r-1}{m_1-1} \Isv_{r}(x_1) \bigg[(n-r) \Isv_{n-r}\bigg(\frac{x_2}{x_1}\bigg)+\log\bigg|\frac{x_2}{x_1}\bigg|^2 \Isv_{n-r-1}\bigg(\frac{x_2}{x_1}\bigg)\bigg]
 \\ &\quad\hphantom{=}
{} +\sum_{r=1}^{m_1}(-1)^{m_1-r}\binom{n-r-1}{m_2-1} \Isv_{r}(x_1) \big[(n-r) \Isv_{n-r}({x_2})+\log|{x_2}|^2 \Isv_{n-r-1}({x_2})\big]\Bigg\} ,
\end{split}
\label{eq:depth2}
\end{gather}
where we interpret $\Isv_{m_1}(x_1)=0$ and $\Isv_{m_1,m_2}(x_1,x_2) = 0$ whenever $m_1$ or $m_2$ are negative.
\end{Proposition}

\begin{proof}
The recursion~\eqref{eq:recursion} holds for arbitrary values of the $x_i$. We can now specialise to the case of depth two with
\begin{gather*}
(x_0;x_{1},\ldots,x_{n};x_{n+1}) = \big(0;y_1,\{0\}^{m_1-1},y_2,\{0\}^{m_2-1};1\big) .
\end{gather*}
It is easy to check that in that case \eqref{eq:recursion} substantially simplifies, and only those terms contribute where $(x_i;x_{i+1},\ldots,x_{j};x_{j+1})$ takes one of the following values:
\begin{gather*}
\big(0;y_1,\{0\}^{m_1-1};y_2\big) ,\qquad (y_1;0;0) ,\qquad (0;0;y_2) ,\qquad (y_2;0;0) ,
\\
\big(y_1;\{0\}^{m_1-1},y_2,\{0\}^{\alpha};0\big) ,\qquad 0\le \alpha < m_2-1 ,
\\
\big(0;\{0\}^{\beta},y_2,\{0\}^{m_1-1};1\big) ,\qquad 0\le \beta < m_1-1 ,
\\
\big(y_1;\{0\}^{m_1-1},y_2,\{0\}^{m_1-1};1\big) .
\end{gather*}
We now go through each of these cases in turn.

\medskip
\noindent
\emph{Case} 1: $(x_i;x_{i+1},\ldots,x_{j};x_{j+1}) = \big(0;y_1,\{0\}^{m_1-1};y_2\big)$.
The corresponding term in the sum in~\eqref{eq:recursion} is
\begin{gather*}
m_1 \Isv\big(0;y_2,\{0\}^{m_2-1};1\big) C\big(0;y_1,\{0\}^{m_1-1};y_2\big) = m_1 \Isv_{m_2}(y_2) C\bigg(0;\frac{y_1}{y_2},\{0\}^{m_1-1};1\bigg) .
\end{gather*}
Using the same argument as in Proposition~\ref{prop:Icsv_d1}, we find
\begin{gather}\label{eq:C_Isv}
C\big(0;y,\{0\}^{m_1-1};1\big) = \Isv_{m_1}(y) + \frac{1}{m_1} \log|y|^2 \Isv_{m_1-1}(y) .
\end{gather}
Hence, we have
\begin{gather*}
m_1 \Isv_{m_2}(y_2) C\bigg(0;\frac{y_1}{y_2},\{0\}^{m_1-1};1\bigg)
 =m_1 \Isv_{m_2}(y_2) \Isv_{m_1}\bigg(\frac{y_1}{y_2}\bigg) + \log\bigg|\frac{y_1}{y_2}\bigg|^2 \Isv_{m_2}(y_2) \Isv_{m_1-1}\bigg(\frac{y_1}{y_2}\bigg) .
\end{gather*}
These match precisely the first and fourth terms in square brackets in~\eqref{eq:depth2}.

\medskip
\noindent
\emph{Case} 2: $(x_i;x_{i+1},\ldots,x_{j};x_{j+1}) = (y_1;0;0)$ or $(0;0;y_2)$.
The sum of these two contributions~is
\begin{gather*}
\Isv\big(0;y_1,\{0\}^{m_1-2},y_2,\{0\}^{m_1-1};1\big) C(y_1;0;0)
 +\Isv\big(0;y_1,\{0\}^{m_1-2},y_2,\{0\}^{m_1-1};1\big) C(0;0;y_2)
 \\ \qquad
{} =-\log|y_1|^2 \Isv_{m_1-1,m_2}(y_1,y_2) + \log|y_2|^2 \Isv_{m_1-1,m_2}(y_1,y_2)
\\ \qquad
{} =\log\bigg|\frac{y_2}{y_1}\bigg|^2 \Isv_{m_1-1,m_2}(y_1,y_2) .
\end{gather*}
This matches the third term in~\eqref{eq:depth2}.

\medskip
\noindent
\emph{Case} 3: $(x_i;x_{i+1},\ldots,x_{j};x_{j+1}) = (y_2;0;0)$. We get
\begin{gather*}
\Isv\big(0;y_1,\{0\}^{m_1-1},y_2,\{0\}^{m_1-2};1\big) C(y_2;0;0)
=-\log|{y_2}|^2 \Isv_{m_1,m_2-1}(y_1,y_2) .
\end{gather*}
This matches the second term in~\eqref{eq:depth2}.

\medskip
\noindent
\emph{Case} 4: $(x_i;x_{i+1},\ldots,x_{j};x_{j+1}) = \big(y_1;\{0\}^{m_1-1},y_2,\{0\}^{\alpha};0\big)$, $0\le \alpha < m_2-1$. We sum up the contributions for different values of $\alpha$ to get
\begin{gather*}
\sum_{\alpha=0}^{m_2-1}(m_1+\alpha) \Isv\big(0;y_1,\{0\}^{m_2-1-\alpha};1\big) C\big(y_1;\{0\}^{m_1-1},y_2,\{0\}^{\alpha};0\big)
\\ \qquad
{} =\sum_{\alpha=0}^{m_2-1}(-1)^{m_1+\alpha} (m_1+\alpha) \Isv_{m_1-\alpha}(y_1) C\big(0;\{0\}^{\alpha},y_2,\{0\}^{m_1-1};y_1\big) .
\end{gather*}
To proceed, we note that $C\big(0;\{0\}^{\alpha},y_2,\{0\}^{m_1-1};y_1\big)$ satisfies Proposition~\ref{prop:unshuffling} with $\Icsv\to C$ ($\Icsv$ is obtained from $C$ by taking the real or imaginary part). Hence
\begin{gather*}
\sum_{\alpha=0}^{m_2-1}(-1)^{m_1+\alpha} (m_1+\alpha) \Isv_{m_1-\alpha}(y_1) C\big(0;\{0\}^{\alpha},y_2,\{0\}^{m_1-1};y_1\big)
\\ \qquad
{} =\sum_{\alpha=0}^{m_2-1}(-1)^{m_1} (m_1+\alpha) \binom{m_1+\alpha-1}{\alpha}\Isv_{m_1-\alpha}(y_1) C\big(0;y_2,\{0\}^{m_1+\alpha-1};y_1\big) .
\end{gather*}
After using~\eqref{eq:C_Isv} and reindexing the sum via $\alpha = m_2-r$, this matches all the terms in the first sum, except for the term $r=1$.

\medskip
\noindent
\emph{Case} 5: $(x_i;x_{i+1},\ldots,x_{j};x_{j+1}) = \big(0;\{0\}^{\beta},y_2,\{0\}^{m_2-1};0\big)$, $0\le \beta < m_1-1$. We proceed in the same way as for Case 4. We find
\begin{gather*}
\sum_{\beta=0}^{m_1-1}(n-\beta-1) \Isv\big(0;y_1,\{0\}^{\beta};1\big)
\\ \qquad
{} =\sum_{\beta=0}^{m_1-1}(-1)^{m_1} (n-\beta-1) \binom{n-\beta-2}{m_1-\beta-1}\Isv_{\beta+1}(y_1) C\big(0;y_2,\{0\}^{n-\beta-2};1\big) .
\end{gather*}
After using~\eqref{eq:C_Isv} and reindexing the sum via $\beta=r-1$, this matches all the terms in the second sum, except for the term with $r=1$.

\medskip
\noindent
\emph{Case} 6: $(x_i;x_{i+1},\ldots,x_{j};x_{j+1}) = \big(y_1;\{0\}^{m_1-1},y_2,\{0\}^{m_2-1};1\big)$. We find, using the path composition and reversal formulas
\begin{gather*}
(n-1) \Isv_{1}(y_1) C\big(y_1;\{0\}^{m_1-1},y_2,\{0\}^{m_2-1};1\big)
\\ \qquad
{} = (n\!-\!1) \Isv_{1}(y_1) \big[C\big(0;\{0\}^{m_1-1},y_2,\{0\}^{m_2-1};1\big)
 \!-\! (-1)^n C\big(0;\{0\}^{m_2-1},y_2,\{0\}^{m_1-1};y_1\big)\big]
 \\ \qquad
 {}= (n-1) \Isv_{1}(y_1) \Bigg[(-1)^{m_1-1} \binom{n-2}{m_1-1} C\big(0;y_2,\{0\}^{n-2};1\big)
 \\ \qquad\hphantom{=(n-1) \Isv_{1}(y_1) \Bigg[}
 {}- (-1)^{m_1} \binom{n-2}{m_2-1} C\bigg(0;\frac{y_2}{y_1},\{0\}^{n-2};1\bigg)\Bigg] .
\end{gather*}
Using~\eqref{eq:C_Isv} we recover the terms with $r=1$ in each of the two sums.
\end{proof}

The functions $\Icsv_{m_1,m_2}(x_1,x_2)$ are real-analytic single-valued functions. For $x_1=x_2=1$, they reduce to zeta values:
\begin{Proposition}\label{prop:depth2_at_1}
For $m_1,m_2>0$, we have
\begin{gather*}
\Icsv_{m_1,m_2}(1,1) = \begin{cases}
c_{m_1m_2} \zeta_{m_1+m_2}, & m_1+m_2\ \text{odd},
\\
0 , & m_1+m_2\ \text{even},
\end{cases}
\end{gather*}
with
\begin{gather*}
c_{m_1m_2} = (-1)^{m_2} \binom{m_1+m_2}{m_1}-1 .
\end{gather*}
\end{Proposition}
\begin{proof}
For $n=m_1+m_2$ even, $\Icsv_{m_1,m_2}(1,1)$ vanishes, because it is the imaginary part of a real number.
If~$n$ is odd, we need to distinguish two cases depending on the parity of $m_1$ and $m_2$.
If~$m_1$ is odd (and thus $m_2$ is even), we have
\begin{gather*}
\Icsv_{m_1,m_2}(1,1) = \Re_n \big[\cR\big(\Imot(0;1,\{0\}^{m_1-1},1,\{0\}^{m_2-1};1)\big)\big] = \cR\big(\zeta^{\mot}_{m_2,m_1}\big) ,
\end{gather*}
where $\zeta^{\mot}_{m_2,m_1}$ is the motivic lift of the double zeta value
\begin{gather*}
\zeta_{m_2,m_1} = \sum_{1\le k_1<k_2}\frac{1}{k_1^{m_1}k_2^{m_2}} .
\end{gather*}
For $n$ odd, $m_2>0$ even and $m_1>1$ odd we have, with $n=2N+1$~\cite{borwein},
\begin{gather*}
\zeta^{\mot}_{m_2,m_1} = \zeta^{\mot}_{m_2} \zeta^{\mot}_{m_1}+\frac{1}{2} \bigg[\binom{n}{m_2}-1\bigg] \zeta^{\mot}_{n}-\sum_{r=1}^{N-1}\bigg[\binom{2r}{a-1}+\binom{2r}{b-1}\bigg] \zeta^{\mot}_{2r+1} \zeta^{\mot}_{n-1-2r} ,
\end{gather*}
while for $m_2$ even and $m_1=1$, we have
\begin{gather*}
\zeta^{\mot}_{m_2,1} =\frac{m_2}{2} \zeta^{\mot}_{m_2+1}-\frac{1}{2}\sum_{r=1}^{m_2-2}\zeta^{\mot}_{r+1} \zeta^{\mot}_{m_2-r} .
\end{gather*}
Hence, $m_2>0$ even and $m_1>0$ odd, we have
\begin{gather*}
\cR\big(\zeta^{\mot}_{m_2,m_1}\big) = \bigg[\binom{n}{m_2}-1\bigg] \zeta_{n} = c_{m_1,m_2} \zeta_n .
\end{gather*}

If $m_1>0$ is even and $m_2>0$ is odd, we start form the well-known stuffle identity for (motivic) MPL's:
\begin{gather*}
\Imot_{m_1}\bigg(\frac{x_1}{x_2}\bigg) \Imot_{m_2}(x_2) = \Imot_{m_1,m_2}(x_1,x_2) + \Imot_{m_2,m_1}\bigg(x_1,\frac{x_1}{x_2}\bigg) - \Imot_{m_1+m_2}(x_1) ,
\end{gather*}
to obtain
\begin{gather*}%\label{eq:clean_stuffle}
\Icsv_{m_1,m_2}(x_1,x_2) = \Icsv_{m_1+m_2}(x_1) - \Icsv_{m_2,m_1}\bigg(x_1,\frac{x_1}{x_2}\bigg) .
\end{gather*}
Hence
\begin{gather*}
\Icsv_{m_1,m_2}(1,1) = \Icsv_{n}(1) - \Icsv_{m_2,m_1}(1,1)
=-2\zeta_{n} - \left[\binom{n}{m_2}-1\right] \zeta_{n} = c_{m_2,m_1} \zeta_n . \tag*{\qed}
\end{gather*}\renewcommand{\qed}{}
\end{proof}

We can also obtain the inversion relation in depth 2 and for all weights:
\begin{Proposition}\label{prop:inversion_depth_2}
With $n=m_1+m_2$, we have
\begin{align*}
\Icsv_{m_1,m_2}\bigg(\frac1{x_1},\frac1{x_2}\bigg) = {}&(-1)^{n} \Icsv_{m_1,m_2}(x_1,x_2) - (-1)^{m_2} \binom{n-1}{m_1} \Icsv_{n}(x_2)
\\
& + (-1)^{m_1} \binom{n-1}{m_2} \Icsv_{n}\left(\frac{x_1}{x_2}\right) - (-1)^{n} \Icsv_{n}(x_1) .
\end{align*}
\end{Proposition}

\begin{proof}This identity was shown to hold modulo products and up to a constant in~\cite{CharltonPhD} (more precisely, it was shown to hold modulo products at symbol level). It is thus sufficient to show that Proposition~\ref{prop:inversion_depth_2} holds at one point, e.g., $x_1=x_2=1$.

For even $n$, both sides of the relation vanish identically at $x_1=x_2=0$, and so the constant is zero. For odd $n$, the right-hand side evaluated at $x_1=x_2=1$ gives
\begin{gather*}
-\Icsv_{m_1,m_2}(1,1) - (-1)^{m_2} \binom{n-1}{m_1} \Icsv_{n}(1)- (-1)^{m_2} \binom{n-1}{m_2} \Icsv_{n}(1) +\Icsv_{n}(1)
\\ \qquad
{} =\zeta_{n}\left[-(-1)^{m_2}\binom{n}{m_1}+1+2(-1)^{m_2} \binom{n-1}{m_1} +2(-1)^{m_2} \binom{n-1}{m_2} -2\right]
 \\ \qquad
{} =\zeta_{n}\left[(-1)^{m_2}\binom{n}{m_1}-1\right] =\Icsv_{m_1,m_2}(1,1) ,
\end{gather*}
where we used the recursion for the binomial coefficients:
\begin{gather*}
\binom{n-1}{m_1} +\binom{n-1}{m_2} = \binom{n-1}{m_1} +\binom{n-1}{m_1-1} = \binom{n}{m_1} . \tag*{\qed}
\end{gather*}
\renewcommand{\qed}{}
\end{proof}

\begin{Corollary}%\label{cor:odd_root_unity}
Let $\xi_N={\rm e}^{2\pi {\rm i}/N}$, $a$, $b$, $m_1$, $m_2$ integers with $n:=m_1+m_2$ odd and $m_1,m_2>0$. Then
\begin{align*}
\Icsv_{m_1,m_2}(\xi_N^a,\xi_N^b) ={}& -\Cl_n\bigg(\frac{2\pi a}{N}\bigg) - (-1)^{m_1}\binom{n-1}{m_1} \Cl_n\bigg(\frac{2\pi b}{N}\bigg)
\\
{} &+(-1)^{m_2}\binom{n-1}{m_2} \Cl_n\bigg(\frac{2\pi (a-b)}{N}\bigg) .
\end{align*}
\end{Corollary}
\begin{proof}
For odd $n$, we have
\begin{gather*}
\Icsv_{m_1,m_2}\big(\xi_N^{-a},\xi_N^{-b}\big) = \Icsv_{m_1,m_2}\big(\xi_N^a,\xi_N^b\big) ,
\end{gather*}
and we see from the inversion relation and Corollary~\ref{prop:Icsv_root_unity} that $\Icsv_{m_1,m_2}\big(\xi_N^a,\xi_N^b\big)$ can be written as a linear combination of Clausen values.
\end{proof}

\subsection{Results in depth 3}

We recall that the explicit version of the parity theorem given by Panzer \cite{panzer_parity} in depth~3 is as follows. Write
\begin{gather*}
\li_{m_1,m_2}(x,y)= I_{m_1,m_2}\big((xy)^{-1},z^{-1}\big) ,\\
\li_{m_1,m_2,m_3}(x,y,z) = I_{m_1,m_2,m_3}\big((xyz)^{-1} ,(yz)^{-1},z^{-1}\big)
\end{gather*}
in terms of the iterated integral notation.
Let $m_1,m_2,m_3 \in \mathbb{Z}_{>0}$, then on the simply-connected domain $\mathbb{C}^3 \setminus \bigcup_{1 \leq i \leq j \leq 3} \{ z \colon z_i z_{i+1} \cdots z_j \in [0, \infty) \}$, which avoids the branch cuts $z_3, z_2 z_3, z_1 z_2 z_3 \in [0, \infty)$ of any the terms $\log(-z_i \cdots z_3)$ therein, the following identity holds,
\begin{gather*}
\li_{m_1,m_2,m_3}(z_1,z_2,z_3) + (-1)^{m_1 + m_2 + m_3} \li_{m_1,m_2,m_3}\big(z_1^{-1},z_2^{-1},z_3^{-1}\big)
\\ \qquad
{}=\li_{m_1}(z_1) \big( \li_{m_2,m_3}(z_2,z_3) - (-1)^{m_2 + m_3} \li_{m_2,m_3}\big(z_2^{-1},z_3^{-1}\big) \big)
\\ \qquad\hphantom{=}
{}- \li_{m_1+m_2,m_3}(z_1 z_2, z_3)+ \li_{m_2,m_1}(z_2,z_1) \ber{m_3}{z_3}
+ \li_{m_2+m_3,m_1}(z_2 z_3, z_1)
\\ \qquad\hphantom{=}
{}- \sum_{\mathclap{\mu+\nu+s= m_2}}\ber{s}{z_1 z_2 z_3}
\binom{-m_3}{\mu}\binom{-m_1}{\nu}\li_{m_3+\mu}(z_3^{-1})
\li_{m_1+\nu}(z_1)(-1)^{m_3+\mu}
\\ \qquad\hphantom{=}
{}- \sum_{\mathclap{\mu+\nu+s=m_3}} \ber{s}{z_1 z_2 z_3}
\binom{-m_2}{\mu}\binom{-m_1}{\nu}\li_{m_2+\mu,m_1+\nu}(z_2,z_1)
\\ \qquad\hphantom{=}
{}- \sum_{\mathclap{\mu+\nu+s=m_1}}\ber{s}{z_1 z_2 z_3}
 \binom{-m_2}{\mu} \binom{-m_3}{\nu} \li_{m_2+\mu,m_3+\nu}\big(z_2^{-1},z_3^{-1}\big) (-1)^{m_2+\mu+m_3+\nu},
\end{gather*}
where
\begin{gather*}
 \ber{n}{z} \mathop{:=} \frac{(2\pi {\rm i})^n}{n!} B_n\bigg(\frac{1}{2} + \frac{\log(-z)}{2 \pi {\rm i}} \bigg)
\end{gather*}
and $B_n(z)$ is the Bernoulli polynomial defined in \eqref{eqn:bernoulli}.

We observe that if $s > 0$ in any of the above sums, then the resulting summand is a product, and so vanishes after passing to the clean single-valued functions $\licsv$. Whereas when $s = 0$, the Bernoulli factor is simply $\ber{0}{z} = 1$. Moreover, the first sum is in fact always a product, as are the first and third terms. After passage to the clean single-valued functions, we obtain the following version of the parity theorem in depth 3, for the clean single-valued functions.

\begin{Proposition}
 Let $m_1, m_2, m_3 \in \mathbb{Z}_{>0}$, then we have
\begin{gather*}
\licsv_{m_1,m_2,m_3}(z_1,z_2,z_3) + (-1)^{m_1 + m_2 + m_3} \licsv_{m_1,m_2,m_3}\big(z_1^{-1},z_2^{-1},z_3^{-1}\big)
\\ \qquad
= - \licsv_{m_1+m_2,m_3}(z_1 z_2, z_3)	+ \licsv_{m_2+m_3,m_1}(z_2 z_3, z_1)
\\ \qquad\hphantom{=}
{}-\sum_{\mu+\nu=m_3}\binom{-m_2}{\mu}\binom{-m_1}{\nu}\licsv_{m_2+\mu,m_1+\nu}(z_2,z_1)
\\ \qquad\hphantom{=}
{}-\sum_{\mu+\nu=m_1} \binom{-m_2}{\mu}\binom{-m_3}{\nu}
 \licsv_{m_2+\mu,m_3+\nu}\big(z_2^{-1},z_3^{-1}\big) (-1)^{m_2+\mu+m_3+\nu}.
\end{gather*}
\end{Proposition}

\section[Clean single-valued Nielsen polylogarithm \protect{S\_\{n,2\}(x)}]
 {Clean single-valued Nielsen polylogarithm $\boldsymbol{S_{n,2}(x)}$}
\label{sec:nielsen}

In this section, we carry out the requisite calculations necessary to explicitly write the single-valued Nielsen polylogarithm $S_{n,2}$, and the clean version thereof. This provides the missing derivation for a formula for $\sv S_{n,2}(x)$ already stated in \cite{Charlton:2019gvp} (the main properties of which were however verified therein).

We recall first the definition of the Nielsen polylogarithm.
\begin{Definition}[Nielsen polylogarithm $S_{n,p}$]
For integers $n, p \geq 0$ the {\it Nielsen polyloga\-rithm $S_{n,p}(x)$} is defined in terms of iterated integrals as follows
\begin{gather*}
S_{n,p}(x) = (-1)^p I(0; \{1\}^p, \{0\}^n; x).
\end{gather*}
Correspondingly the \emph{motivic Nielsen polylogarithm} is
\begin{gather*}
S^\mot_{n,p}(x) = (-1)^p I^\mot(0; \{1\}^p, \{0\}^n; x).
\end{gather*}
\end{Definition}

Note that, for $p = 1$, we have $S_{n,1}(x) = \li_{n+1}(x)$, which recovers the classical polylogarithm as a special case.

\subsection[Coproduct of \protect{S\textasciicircum{}\{dr\}\_\{n,2\}(x)}]
 {Coproduct of $\boldsymbol{S^\dr_{n,2}(x)}$}

We start by computing the coproduct $\Delta S^\dr_{n,2}(x) = \Delta I^\dr(0; 1, 1, \{0\}^n; x)$. We first consider which terms will contribute, and which can be ignored. We recall that each term in the coproduct $\Delta I^\dr(x_0; x_1,\ldots,x_n; x_{n+1})$ is obtained by selecting a subset $S$ of points $x_0, x_1, \ldots, x_n, x_{n+1}$, where $x_0$ and $x_{n+1}$ are always to be included (cf.~\eqref{eq:coaction}, where~$i_j$ indexes the subset $S$ in the coaction). The intervals between these points contribute a product of integrals on the right hand side of the coproduct, the subset itself contributes a single integral of these points on the left hand side of the coproduct. These terms are often pictorially represented with the mnemonic of arcs joining the vertices of a semi-circular polygon (see~\cite[Theorem~1.2]{Goncharov:2005sla} and the remarks thereafter).

If we do not include $x_1$ as part of a term in the reduced coproduct $\Delta'$, then the first point included is either $x_2$ which contributes $I^\dr(x_0; x_1; x_2) = I^\dr(0; 1; 1) = 0$ {(with the upper integration limit a tangential base-point at $1$, giving the required regularisation)} to the product side, or is $x_{k}$ for $k \geq 3$, which contributes $I^\dr(x_0; x_1,\ldots,x_{k-1}; x_k) = I^\dr(0; 1, 1, \{0\}^\bullet; 0) = 0$ to the product side:
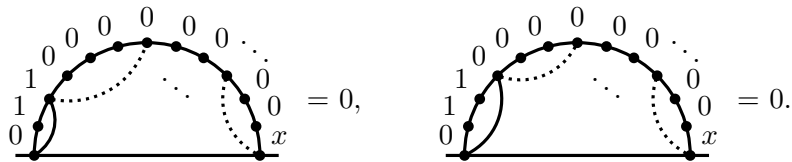
\begin{center}
\begin{tikzpicture}[style = {very thick}, scale = 0.5, baseline={([yshift=-1em]current bounding box.center)}]
	\draw (3,0) arc [start angle=0, end angle = 180, radius=3]
	node[pos=0.000, name=x12, inner sep=0pt] {}
	node[pos=0.083, name=x11, inner sep=0pt] {}
	node[pos=0.167, name=x10, inner sep=0pt] {}
	node[pos=0.250, name=x9, inner sep=0pt] {}
	node[pos=0.333, name=x8, inner sep=0pt] {}
	node[pos=0.417, name=x7, inner sep=0pt] {}
	node[pos=0.500, name=x6, inner sep=0pt] {}
	node[pos=0.583, name=x5, inner sep=0pt] {}
	node[pos=0.667, name=x4, inner sep=0pt] {}
	node[pos=0.750, name=x3, inner sep=0pt] {}
	node[pos=0.833, name=x2, inner sep=0pt] {}
	node[pos=0.917, name=x1, inner sep=0pt] {}
	node[pos=1.0, name=x0, inner sep=0pt] {};
	\draw (-3.5,0) -- (3.5,0);
	\filldraw
	(x0) circle [radius=0.1, fill=black, red] node[above left] {$0$}
	(x1) circle [radius=0.1, fill=black] node[above left] {$1$}
	(x2) circle [radius=0.1, fill=black] node[above left] {$1$}
	(x3) circle [radius=0.1, fill=black] node[above left] {$0$}
	(x4) circle [radius=0.1, fill=black] node[above left] {$0$}
	(x5) circle [radius=0.1, fill=black] node[above left, xshift=2] {$0$}
	(x6) circle [radius=0.1, fill=black] node[above, yshift=2] {$0$}
	(x7) circle [radius=0.1, fill=black] node[above right, xshift=-2] {$0$}
	(x8) circle [radius=0.1, fill=black] node[above right] {$0$}
	(x9) circle [radius=0.1, fill=black] node[above right] {$\ddots$}
	(x10) circle [radius=0.1, fill=black] node[above right] {$0$}
	(x11) circle [radius=0.1, fill=black] node[above right] {$0$}
	(x12) circle [radius=0.1, fill=black] node[above right] {$x$};
	\draw (x7) node[below] {$\ddots$};
	\draw [solid, bend right=40] (x0) edge (x2);
	\draw [dotted, bend right=40] (x2) edge (x6);
	\draw [dotted, bend right=40] (x9) edge (x12);
	\end{tikzpicture} = 0, \qquad
	\begin{tikzpicture}[style = {very thick}, scale = 0.5,baseline={([yshift=-1em]current bounding box.center)}]
	\draw (3,0) arc [start angle=0, end angle = 180, radius=3]
	node[pos=0.000, name=x12, inner sep=0pt] {}
	node[pos=0.083, name=x11, inner sep=0pt] {}
	node[pos=0.167, name=x10, inner sep=0pt] {}
	node[pos=0.250, name=x9, inner sep=0pt] {}
	node[pos=0.333, name=x8, inner sep=0pt] {}
	node[pos=0.417, name=x7, inner sep=0pt] {}
	node[pos=0.500, name=x6, inner sep=0pt] {}
	node[pos=0.583, name=x5, inner sep=0pt] {}
	node[pos=0.667, name=x4, inner sep=0pt] {}
	node[pos=0.750, name=x3, inner sep=0pt] {}
	node[pos=0.833, name=x2, inner sep=0pt] {}
	node[pos=0.917, name=x1, inner sep=0pt] {}
	node[pos=1.0, name=x0, inner sep=0pt] {};
	\draw (-3.5,0) -- (3.5,0);
	\filldraw
	(x0) circle [radius=0.1, fill=black, red] node[above left] {$0$}
	(x1) circle [radius=0.1, fill=black] node[above left] {$1$}
	(x2) circle [radius=0.1, fill=black] node[above left] {$1$}
	(x3) circle [radius=0.1, fill=black] node[above left] {$0$}
	(x4) circle [radius=0.1, fill=black] node[above left] {$0$}
	(x5) circle [radius=0.1, fill=black] node[above left, xshift=2] {$0$}
	(x6) circle [radius=0.1, fill=black] node[above, yshift=2] {$0$}
	(x7) circle [radius=0.1, fill=black] node[above right, xshift=-2] {$0$}
	(x8) circle [radius=0.1, fill=black] node[above right] {$0$}
	(x9) circle [radius=0.1, fill=black] node[above right] {$\ddots$}
	(x10) circle [radius=0.1, fill=black] node[above right] {$0$}
	(x11) circle [radius=0.1, fill=black] node[above right] {$0$}
	(x12) circle [radius=0.1, fill=black] node[above right] {$x$};
	\draw (x7) node[below] {$\ddots$};
	\draw [solid, bend right=40] (x0) edge (x3);
	\draw [dotted, bend right=40] (x3) edge (x6);
	\draw [dotted, bend right=40] (x9) edge (x12);
	\end{tikzpicture} = 0.
\end{center}
Hence we must include $x_1$ in the subset.	 If this is the entire subset, we obtain the following term in $\Delta S^\dr_{n,2}(x)$
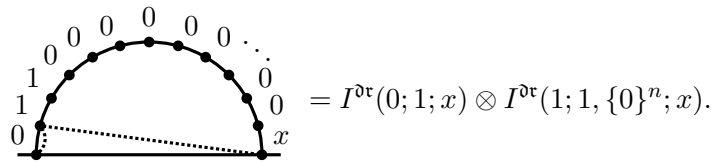
\begin{center}
	
	\begin{tikzpicture}[style = {very thick}, scale = 0.5,baseline={([yshift=-1em]current bounding box.center)}]
	\draw (3,0) arc [start angle=0, end angle = 180, radius=3]
	node[pos=0.000, name=x12, inner sep=0pt] {}
	node[pos=0.083, name=x11, inner sep=0pt] {}
	node[pos=0.167, name=x10, inner sep=0pt] {}
	node[pos=0.250, name=x9, inner sep=0pt] {}
	node[pos=0.333, name=x8, inner sep=0pt] {}
	node[pos=0.417, name=x7, inner sep=0pt] {}
	node[pos=0.500, name=x6, inner sep=0pt] {}
	node[pos=0.583, name=x5, inner sep=0pt] {}
	node[pos=0.667, name=x4, inner sep=0pt] {}
	node[pos=0.750, name=x3, inner sep=0pt] {}
	node[pos=0.833, name=x2, inner sep=0pt] {}
	node[pos=0.917, name=x1, inner sep=0pt] {}
	node[pos=1.0, name=x0, inner sep=0pt] {};
	\draw (-3.5,0) -- (3.5,0);
	\filldraw
	(x0) circle [radius=0.1, fill=black, red] node[above left] {$0$}
	(x1) circle [radius=0.1, fill=black] node[above left] {$1$}
	(x2) circle [radius=0.1, fill=black] node[above left] {$1$}
	(x3) circle [radius=0.1, fill=black] node[above left] {$0$}
	(x4) circle [radius=0.1, fill=black] node[above left] {$0$}
	(x5) circle [radius=0.1, fill=black] node[above left, xshift=2] {$0$}
	(x6) circle [radius=0.1, fill=black] node[above, yshift=2] {$0$}
	(x7) circle [radius=0.1, fill=black] node[above right, xshift=-2] {$0$}
	(x8) circle [radius=0.1, fill=black] node[above right] {$0$}
	(x9) circle [radius=0.1, fill=black] node[above right] {$\ddots$}
	(x10) circle [radius=0.1, fill=black] node[above right] {$0$}
	(x11) circle [radius=0.1, fill=black] node[above right] {$0$}
	(x12) circle [radius=0.1, fill=black] node[above right] {$x$};
	\draw [densely dotted, bend right=40] (x0) edge (x1);
	\draw [densely dotted, bend right=0] (x1) edge (x12);
	\end{tikzpicture}
	$= I^\dr(0; 1; x) \otimes I^\dr(1; 1, \{0\}^n; x)$.
\end{center}
Suppose now we take some point $x_j$ as the next point, and some point $x_k$, $k \geq j$ as the last point in the subset. We notice now that we must also take every point $x_i$ with $j \leq i \leq k$, else we will contribute a term $I^\dr(0; 0; 0) = 0$ to the product side. Hence the only other terms which contribute are of the form
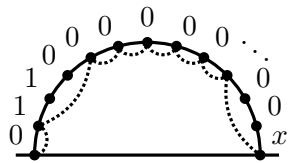
\begin{center}
	\begin{tikzpicture}[style = {very thick}, scale = 0.5,baseline={([yshift=-1em]current bounding box.center)}]
	\draw (3,0) arc [start angle=0, end angle = 180, radius=3]
	node[pos=0.000, name=x12, inner sep=0pt] {}
	node[pos=0.083, name=x11, inner sep=0pt] {}
	node[pos=0.167, name=x10, inner sep=0pt] {}
	node[pos=0.250, name=x9, inner sep=0pt] {}
	node[pos=0.333, name=x8, inner sep=0pt] {}
	node[pos=0.417, name=x7, inner sep=0pt] {}
	node[pos=0.500, name=x6, inner sep=0pt] {}
	node[pos=0.583, name=x5, inner sep=0pt] {}
	node[pos=0.667, name=x4, inner sep=0pt] {}
	node[pos=0.750, name=x3, inner sep=0pt] {}
	node[pos=0.833, name=x2, inner sep=0pt] {}
	node[pos=0.917, name=x1, inner sep=0pt] {}
	node[pos=1.0, name=x0, inner sep=0pt] {};
	\draw (-3.5,0) -- (3.5,0);
	\filldraw
	(x0) circle [radius=0.1, fill=black, red] node[above left] {$0$}
	(x1) circle [radius=0.1, fill=black] node[above left] {$1$}
	(x2) circle [radius=0.1, fill=black] node[above left] {$1$}
	(x3) circle [radius=0.1, fill=black] node[above left] {$0$}
	(x4) circle [radius=0.1, fill=black] node[above left] {$0$}
	(x5) circle [radius=0.1, fill=black] node[above left, xshift=2] {$0$}
	(x6) circle [radius=0.1, fill=black] node[above, yshift=2] {$0$}
	(x7) circle [radius=0.1, fill=black] node[above right, xshift=-2] {$0$}
	(x8) circle [radius=0.1, fill=black] node[above right] {$0$}
	(x9) circle [radius=0.1, fill=black] node[above right] {$\ddots$}
	(x10) circle [radius=0.1, fill=black] node[above right] {$0$}
	(x11) circle [radius=0.1, fill=black] node[above right] {$0$}
	(x12) circle [radius=0.1, fill=black] node[above right] {$x$};
	\draw [densely dotted, bend right=70] (x0) edge (x1);
	\draw [densely dotted, bend right=30] (x1) edge (x4);
	\draw [densely dotted, bend right=70] (x4) edge (x5);
	\draw [densely dotted, bend right=70] (x5) edge (x6);
	\draw [densely dotted, bend right=70] (x6) edge (x7);
	\draw [densely dotted, bend right=70] (x7) edge (x8);
	\draw [densely dotted, bend right=70] (x8) edge (x9);
	\draw [densely dotted, bend right=30] (x9) edge (x12);
	\end{tikzpicture}
\end{center}
Depending on whether $x_2$ is part of the subset or not, the terms have different forms as functions, namely
\begin{center}
	\begin{tikzpicture}[style = {very thick}, scale = 0.5,baseline={([yshift=-1em]current bounding box.center)}]
	\draw (3,0) arc [start angle=0, end angle = 180, radius=3]
	node[pos=0.000, name=x12, inner sep=0pt] {}
	node[pos=0.083, name=x11, inner sep=0pt] {}
	node[pos=0.167, name=x10, inner sep=0pt] {}
	node[pos=0.250, name=x9, inner sep=0pt] {}
	node[pos=0.333, name=x8, inner sep=0pt] {}
	node[pos=0.417, name=x7, inner sep=0pt] {}
	node[pos=0.500, name=x6, inner sep=0pt] {}
	node[pos=0.583, name=x5, inner sep=0pt] {}
	node[pos=0.667, name=x4, inner sep=0pt] {}
	node[pos=0.750, name=x3, inner sep=0pt] {}
	node[pos=0.833, name=x2, inner sep=0pt] {}
	node[pos=0.917, name=x1, inner sep=0pt] {}
	node[pos=1.0, name=x0, inner sep=0pt] {};
	\draw (-3.5,0) -- (3.5,0);
	\filldraw
	(x0) circle [radius=0.1, fill=black, red] node[above left] {$0$}
	(x1) circle [radius=0.1, fill=black] node[above left] {$1$}
	(x2) circle [radius=0.1, fill=black] node[above left] {$1$}
	(x3) circle [radius=0.1, fill=black] node[above left] {$0$}
	(x4) circle [radius=0.1, fill=black] node[above left, xshift=3, yshift=-1] {$\iddots$}
	(x5) circle [radius=0.1, fill=black] node[above left, xshift=2] {$0$}
	(x6) circle [radius=0.1, fill=black] node[above, yshift=2] {$0$}
	(x7) circle [radius=0.1, fill=black] node[above right, xshift=-2] {$0$}
	(x8) circle [radius=0.1, fill=black] node[above right] {$0$}
	(x9) circle [radius=0.1, fill=black] node[above right] {$0$}
	(x10) circle [radius=0.1, fill=black] node[above right] {$\ddots$}
	(x11) circle [radius=0.1, fill=black] node[above right] {$0$}
	(x12) circle [radius=0.1, fill=black] node[above right] {$x$};
	\draw [densely dotted, bend right=70] (x0) edge (x1);
	\draw [densely dotted, bend right=70] (x1) edge (x2);
	\draw [densely dotted, bend right=70] (x2) edge (x3);
	\draw [densely dotted, bend right=70] (x3) edge (x4);
	\draw [densely dotted, bend right=70] (x4) edge (x5);
	\draw [densely dotted, bend right=30] (x5) edge (x6);
	\draw [densely dotted, bend right=70] (x6) edge (x7);
	\draw [densely dotted, bend right=70] (x7) edge (x8);
	\draw [densely dotted, bend right=30] (x8) edge (x12);
	\end{tikzpicture}
	$	= S^\dr_{n-j,2}(x) \otimes \dfrac{(\log^\dr x)^j}{j!} \qquad \text{for}\quad j = 1, \ldots, n$,
\end{center}\begin{center}
$$
	\begin{tikzpicture}[style = {very thick}, scale = 0.5,baseline={([yshift=-1em]current bounding box.center)}]
	\draw (3,0) arc [start angle=0, end angle = 180, radius=3]
	node[pos=0.000, name=x12, inner sep=0pt] {}
	node[pos=0.083, name=x11, inner sep=0pt] {}
	node[pos=0.167, name=x10, inner sep=0pt] {}
	node[pos=0.250, name=x9, inner sep=0pt] {}
	node[pos=0.333, name=x8, inner sep=0pt] {}
	node[pos=0.417, name=x7, inner sep=0pt] {}
	node[pos=0.500, name=x6, inner sep=0pt] {}
	node[pos=0.583, name=x5, inner sep=0pt] {}
	node[pos=0.667, name=x4, inner sep=0pt] {}
	node[pos=0.750, name=x3, inner sep=0pt] {}
	node[pos=0.833, name=x2, inner sep=0pt] {}
	node[pos=0.917, name=x1, inner sep=0pt] {}
	node[pos=1.0, name=x0, inner sep=0pt] {};
	\draw (-3.5,0) -- (3.5,0);
	\filldraw
	(x0) circle [radius=0.1, fill=black, red] node[above left] {$0$}
	(x1) circle [radius=0.1, fill=black] node[above left] {$1$}
	(x2) circle [radius=0.1, fill=black] node[above left] {$1$}
	(x3) circle [radius=0.1, fill=black] node[above left] {$0$}
	(x4) circle [radius=0.1, fill=black] node[above left, xshift=3, yshift=-1] {$\iddots$}
	(x5) circle [radius=0.1, fill=black] node[above left, xshift=2] {$0$}
	(x6) circle [radius=0.1, fill=black] node[above, yshift=2] {$0$}
	(x7) circle [radius=0.1, fill=black] node[above right, xshift=-2] {$0$}
	(x8) circle [radius=0.1, fill=black] node[above right] {$0$}
	(x9) circle [radius=0.1, fill=black] node[above right] {$0$}
	(x10) circle [radius=0.1, fill=black] node[above right] {$\ddots$}
	(x11) circle [radius=0.1, fill=black] node[above right] {$0$}
	(x12) circle [radius=0.1, fill=black] node[above right] {$x$};
	\draw [densely dotted, bend right=70] (x0) edge (x1);
	\draw [densely dotted, bend right=70] (x1) edge (x5);
	\draw [densely dotted, bend right=30] (x5) edge (x6);
	\draw [densely dotted, bend right=70] (x6) edge (x7);
	\draw [densely dotted, bend right=70] (x7) edge (x8);
	\draw [densely dotted, bend right=30] (x8) edge (x12);
	\end{tikzpicture}
	\begin{aligned} &= -\li^\dr_{n+2-k-j}(x) \otimes \zeta^\dr_k \cdot \frac{(\log^\dr x)^{j}}{j!}
\\ & \quad \text{for}\quad k \geq 1\quad \text{odd},\quad j \geq 0\quad \text{with}\quad k+j \leq n,
	\end{aligned}
$$
\end{center}
where the term $S^\dr_{0,2}(x) = I^\dr(0;1,1;x)$ is interpreted as $\frac{1}{2!} (\log^\dr(1-x))^2$ via the integral representation of $S^\dr_{n,2}(x)$ and the shuffle product thereof. Recall also that the terms $\zeta^\dr_{2k} = 0$ vanish on the right hand factor of the coproduct.

Hence we obtain
\begin{align*}
\Delta S^\dr_{n,2}(x) = {} & 1 \otimes S^\dr_{n,2}(x) + S^\dr_{n,2}(x) \otimes 1 + \log^\dr(1-x) \otimes I^\dr(1; 1, \{0\}^n; x)
\\
& - \sum_{\substack{k=3 \\ \text{$k$ odd}}}^{n+1} \sum_{j=0}^{n-k} \li^\dr_{n+2-k-j}(x) \otimes \zeta^\dr_k \cdot \frac{(\log^\dr x)^j}{j!}
 + \sum_{j=1}^n S^\dr_{n-j,2}(x) \otimes \frac{(\log^\dr x)^j}{j!} .
\end{align*}
We note moreover that via decomposition of paths and the regularisation $I^\mot\big(1; 1, \{0\}^{j-1}; 0\big) = \zeta^\mot_j$ we have
\begin{gather*}
I^\dr(1; 1, \{0\}^n; x) = -\li^\dr_{n+1}(x) + \sum_{k=2}^{n+1} \zeta^\dr_k \frac{(\log^\dr x)^{n+1-k}}{(n+1-k)!} ,
\end{gather*}
which can be substituted into the above, along the additional restriction ``$k$ odd'' as $ \zeta^\dr_{2k}$ does not contribute to the right hand factor of the coproduct. Then the summation in this term can in fact be combined with the double-sum given above, as the $j=n+1-k$ term of the inner sum, since $\log(1-x) = -\li_1(x)$. Overall we obtain (after reversing both $j$ sums for later convenience, and changing $k = k' + 2$)
\begin{align*}
\Delta S^\dr_{n,2}(x) = {}& 1 \otimes S^\dr_{n,2}(x) + S^\dr_{n,2}(x) \otimes 1 - \log^\dr(1-x) \otimes \li^\dr_{n+1}(x)
\\
& - \sum_{\substack{k=1 \\ \text{$k$ odd}}}^{n-1} \sum_{j=1}^{n-k} \li^\dr_{j}(x) \otimes \zeta^\dr_{k+2}\frac{(\log^\dr x)^{n-k-j}}{(n-k-j)!}
 + \sum_{j=0}^{n-1} S^\dr_{j,2}(x) \otimes \frac{(\log^\dr x)^{n-j}}{(n-j)!} .
\end{align*}

\subsection[Antipode of \protect{S\textasciicircum{}\{dr\}\_\{n,2\}(x)}]
{Antipode of $\boldsymbol{S^\dr_{n,2}(x)}$}

Next we compute the antipode; the general recursion says
\begin{gather*}
S\big(S^\dr_{n,2}(x)\big) = -S^\dr_{n,2}(x) - m(\operatorname{id} \otimes S) \Delta' S^\dr_{n,2}(x) .
\end{gather*}
The only terms appearing on the right hand side of the coproduct, whose antipode we recursively need, are $S(\log^\dr x) = -\log^{\dr}x$, $S(\zeta_k^{\dr}) = -\zeta^{\dr}_k${ for $k$ odd}, as odd Riemann zeta values are primitives for $\Delta$, and the previously computed
\begin{align*}
	S\big(\li^\dr_n(x)\big) = -\li^\dr_n(x) - \sum_{k=1}^{n-1}\frac{(-\log^\dr x)^k}{k!}\li^\dr_{n-k}(x)
 = - \sum_{k=1}^{n} \li^\dr_{k}(x) \frac{(-\log^\dr x)^{n-k}}{(n-k)!} ,
\end{align*} as given in \eqref{eq:Delta_Li}. Since we have these we can directly evaluate the antipode $S(S^\dr_{n,2}(x))$ to see
\begin{align*}
S\big(S^\dr_{n,2}(x)\big) = &
- \sum_{j=0}^n S^\dr_{j,2}(x) \frac{(-\log^\dr x)^{n-j}}{(n-j)!}
 - \sum_{j=1}^{n+1} \li^\dr_{j}(x) \log^\dr(1-x) \frac{(-\log^\dr x)^{n+1-j}}{(n+1-j)!} \\
& - \sum_{\substack{k=1 \\ \text{$k$ odd}}}^{n-1} \sum_{j=1}^{n-k} \li^\dr_{j}(x) \zeta^\dr_{k+2} \frac{(-\log^\dr x)^{n-j-k}}{(n-j-k)!} .
\end{align*}

\subsection[Single-valued version \protect{sv\textasciicircum{}m} \protect{S\textasciicircum{}m\_\{n,2\}(x)}]
 {Single-valued version $\boldsymbol{\sv^\mot S^\mot_{n,2}(x)}$}

	From this we are in a position to write the single-valued version via
\begin{gather*}
\sv^\mot S^\mot_{n,2}(x) = \per \circ m (F_\infty \Sigma \otimes \id) \widetilde{\Delta} S^\dr_{n,2}(x).
\end{gather*}
We shall treat each term of the above coproduct in turn.

Firstly
\begin{gather*}
 \per \circ m(F_\infty \Sigma \otimes \id) \big(1 \otimes S^\dr_{n,2}(x) \!- \log^\dr(1-x) \otimes \li^\dr_{n+1}(x)\big)
= S_{n,2}(x)\! - \log(1\!-\bar{x}) \li_{n+1}(x) .
\end{gather*}
More complicated is the term
\begin{gather*}
 \per \circ m(F_\infty \Sigma \otimes \id) \Bigg({-}\sum_{\substack{k=1 \\ \text{$k$ odd}}}^{n-1} \sum_{j=1}^{n-k} \li^\dr_{j}(x) \otimes \zeta^\dr_{k+2} \frac{(\log^\dr x)^{n-k-j}}{(n-k-j)!} \Bigg)
 \\ \qquad
{}= \sum_{\substack{k=1 \\ \text{$k$ odd}}}^{n-1} \sum_{j=1}^{n-k} \Bigg((-1)^j \sum_{\ell=1}^j \li_\ell(\bar{x}) \frac{(-\log\bar{x})^{j-\ell}}{(j-\ell)!} \Bigg) \zeta_{k+2} \frac{\log^{n-k-j}x}{(n-k-j)!} .
\end{gather*}
Now switch the order of summation between $j$ and $\ell$, then set $j' = j - \ell$ in the inner-most sum. We obtain
\begin{gather*}
{}= \sum_{\substack{k=1 \\ \text{$k$ odd}}}^{n-1} \sum_{\ell=1}^{n-k} (-1)^\ell \li_\ell(\bar{x}) \zeta_{k+2}\sum_{j=0}^{n-k-\ell} \frac{(\log\bar{x})^j}{j!} \frac{(\log x)^{n-k-\ell-j}}{(n-k-j-\ell)!} \\
{} = \sum_{\substack{k=1 \\ \text{$k$ odd}}}^{n-1} \sum_{\ell=1}^{n-k} (-1)^\ell \li_\ell(\bar{x}) \zeta_{k+2} \frac{(\log|x|^2 )^{n-k-\ell}}{(n-k-\ell)!} .
\end{gather*}
Here we have written $\log\bar{x} + \log x = \log|x|^2$, and applied the binomial theorem to evaluate the inner-most sum.

Finally, we can compute (note we take $j=n$ in the sum, accounting also for the $S_{n,2}(x) \otimes 1$ term)\vspace{-1ex}
\begin{gather*}
\per \circ m (F_\infty \Sigma \otimes \id) \Bigg( \sum_{j=0}^{n-1} S^\dr_{j,2}(x) \otimes \frac{(\log^\dr x)^{n-j}}{(n-j)!} \Bigg)
\\ \qquad
{}= -\sum_{j=0}^n (-1)^j \Bigg\{ \sum_{\ell=0}^j S_{\ell,2}(\bar{x}) \frac{(-\log\bar{x})^{j-\ell}}{(j-\ell)!}
	- \sum_{\ell=1}^{j+1} \li_\ell(\bar{x}) \log(1 - \bar{x}) \frac{(-\log\bar{x})^{j+1-\ell}}{(j+1-\ell)!}
\\ \qquad\hphantom{= -\sum_{j=0}^n (-1)^j \Bigg\{}
	-\sum_{\substack{k=1 \\ \text{$k$ odd}}}^{j-1} \sum_{\ell=1}^{j-k} \li_\ell(\bar{x}) \zeta_{k+2} \frac{(-\log\bar{x})^{j-\ell-k}}{(j-\ell-k)!} \Bigg\} \frac{(\log x)^{n-j}}{(n-j)!} .
\end{gather*}
Like previously, we interchange the $j$ and $\ell$ sums, or more accurately move the $j$ sum inside the~$\ell$ sum for the third term. In each case, the resulting sum of $j$ can be evaluated as a power of $\log x + \log\bar{x} = \log|x|^2$ via the binomial theorem. We obtain
\begin{align*}
	 = & - \sum_{\ell=0}^n (-1)^\ell S_{\ell,2}(\bar{x}) \frac{\log^{n-\ell}|x|}{(n-\ell)!} + \sum_{\ell=1}^{n+1} (-1)^\ell \li_{\ell}(\bar{x}) \log(1- \bar{x}) \frac{(\log|x|^2)^{n+1 - \ell}}{(n+1-\ell)!} \\
	& + \sum_{\substack{k=1 \\ \text{$k$ odd}}}^{n-1} \sum_{\ell=1}^{n-k} (-1)^\ell \li_\ell(\bar{x}) \zeta_{k+2} \frac{(\log|x|^2 )^{n-k-\ell}}{(n-k-\ell)!} .
\end{align*}
We note the simplification $-(-1)^{\ell+k} = (-1)^\ell$ since $k$ is odd has been used to obtain the last term. Moreover, the resulting term is exactly the same as already obtained above.

Summing the above 3 contributions, and rewriting slightly, gives the following formula for $\sv^\mot S^\mot_{n,2}(x)$, as stated in~\cite[Section~4.3]{Charlton:2019gvp}. Namely
\begin{align*}
\sv^\mot S_{n,2}^\mot(x) = {} & \big( S_{n,2}(x)+ (-1)^{n+1} S_{n,2}(\bar{x}) \big)
 -\log(1-\bar{x}) \big(\li_{n+1}(x) + (-1)^n\li_{n+1}(\bar{x}) \big)
 \\
& -\sum _{j=0}^{n-1} \frac{(-1)^j}{(n-j)!} \log ^{n-j}|x|^2 \big( S_{j,2}(\bar{x}) + \log (1-\bar{x}) \li_{j+1}(\bar{x}) \big)
\\
& +\sum_{\substack{k=1 \\ \text{$k$ odd}}}^{n-1} \sum_{j=1}^{n-k} \frac{2
(-1)^j \zeta_{k+2}}{(n-j-k)!} \li_j(\bar{x}) \log^{n-j-k}|x|^2 .
\end{align*}

\subsection[Clean version of \protect{S\_\{n,2\}(x)}]
 {Clean version of $\boldsymbol{S_{n,2}(x)}$}

We also briefly repeat the calculation of $\Pi S_{n,2}^{\dr}(x)$, which was completed in detail in~\cite{Charlton:2019gvp}. Here we proceed directly via the already calculated coproduct of $S_{n,2}^\dr(x)$, whereas the calculation in \cite{Charlton:2019gvp} was reduced to an analysis of which terms contribute in the infinitesimal coproduct. From~\eqref{eqn:Drecursion}, we have
\begin{gather*}
	D\big(S_{n,2}^\dr(x)\big) = (n+2) S_{n,2}^\dr(x) - m(\id \otimes (Y \cdot \Pi)) \Delta' S^\dr_{n,2} .
\end{gather*}
Again, because of the projector in the second tensor factor, we are free to ignore all products in the coproduct while evaluating this. From the calculation above of $\Delta S^\dr_{n,2}(x)$, we have\vspace{-1ex}
\begin{align*}
	\Delta' S^\dr_{n,2}(x) \equiv {} & -\log^\dr(1-x) \otimes \li_{n+1}^\dr(x) - \sum_{\substack{k=1 \\ \text{$k$ odd}}}^{n-1} \li_{n-k}^\dr(x) \otimes \zeta^\dr_{k+2} \\
		& + S_{n-1,2}^\dr(x) \otimes \log^\dr x \quad \pmod{\text{right-$\otimes$-factor products}} .
\end{align*}
Hence
\begin{align*}
	\Pi S^\dr_{n,2}(x) = {} & S_{n,2}^\dr(x) - \frac{1}{n+2} m(\id \otimes (Y \cdot \Pi)) \big( \Delta' S^\dr_{n,2}(x) \big) \\
	= {} & S_{n,2}^\dr(x) - S_{n-1,2}^\dr(x) \frac{\log^\dr x}{n+2} + \frac{n+1}{n+2} \log^\dr(1-x) \li_{n+1}^\dr(x) \\
	& - \frac{1}{n+2} \log^\dr(1-x) \log^\dr(x) \li_{n}^\dr(x) \\
	& + \sum_{\substack{k=1 \\ \text{$k$ odd}}}^{n-1}\frac{k+2}{n+2} \zeta^\dr_{k+2} \li_{n-k}^\dr(x) .
\end{align*}
Here we have applied the result that $\Pi \li_n^\dr(x) = \li_n^\dr(x) - \frac{1}{n} \log^\dr x \li_{n-1}^\dr(x)$, which can be obtained from \eqref{eqn:PiIn}. Likewise, we have also used that $ \Pi \zeta^\dr_{2k+1} = \zeta^\dr_{2k+1}$, since $\zeta_{2k+1}^{\dr}$ is a primitive for the coproduct.

\subsection[Clean single-valued \protect{S\_\{n,2\}(x)}]
{Clean single-valued $\boldsymbol{S_{n,2}(x)}$}

Application of the period and the single-valued map to the expression $\Pi S_{n,2}^\dr(x)$, using the previous computations of the single-valued versions of $\li_n^\dr(x)$ and $\zeta^\dr_{2k+1}$, and the computation of $S_{n,2}^\dr(x)$ directly above gives us an expression for $\mathcal{R} S_{n,2}^\dr(x)$. We can then define the clean single-valued version of $S_{n,2}^\dr$.
\begin{Definition}
The \emph{clean single-valued Nielsen polylogarithm $S^\mathrm{csv}_{n,2}(x)$} is defined by $ S^\mathrm{csv}_{n,2}(x) := \mathfrak{R}_n \mathcal{R} S_{n,2}^\dr(x)$, where $\mathcal{R} S_{n,2}(x)$ is as follows
\begin{align*}
\mathcal{R} S_{n,2}^\dr(x) ={}& \big(S_{n,2}(x)-(-1)^{n+2} S_{n,2}(\bar{x}) \big)
-\frac{\log |x|^2}{n+2} \big( S_{n-1,2}(x) + \log (1-x) \li_n(x) \big)
\\
& + \frac{1}{n+2} \big((n+1) \log(1-x) - \log(1-\bar{x})\big) \big( \li_{n+1}(x) - (-1)^{n+1} \li_{n+1}(\bar{x}) \big)
\\
& +\!\sum_{j=1}^{n} \!\frac{(-1)^j }{n\!+\!2} \Big\{\!(j+1) S_{j-1,2}(\bar{x})
 \!-\! \big(j \log (1\!-\!x)\!-\! \log(1\!-\!\bar{x})\big) \li_j(\bar{x}) \!\Big\} \frac{\log ^{n-j+1}|x|^2}{(n-j+1)!}
 \\
& + \sum_{\substack{k=1 \\ \text{$k$ odd}}}^{n-1}\frac{2 \zeta_{k+2}}{n+2} \Bigg\{ (k+2) \li_{n-k}(x) + \sum _{j=1}^{n-k} j (-1)^j \frac{\log ^{n-j-k}|x|^2}{ (n-j-k)!} \li_j(\bar{x}) \Bigg\} .
\end{align*}
\end{Definition}

These clean single-valued functions (or more precisely the version without ${\mathfrak R}_n$, which still satisfies Theorem \ref{thm:main}) are used in \cite{Charlton:2019gvp} to obtain numerically verifiable identities and reductions for Nielsen polylogarithms, including numerical results on their special values.

\section{Numerical evaluations}\label{sec:numerical}

In this last section, we derive a few numerical evaluations of depth 2 MPL's through use of the clean single-valued functions. These evaluations would perhaps otherwise not be obtainable in such a straightforward manner. For the sake of simplicity, we restrict to one identity in weight 4 and one in weight 5 obtained from the 2-term symmetry of $I_{3,1}$ relating $I_{3,1}(x,y)$ and $ I_{3,1}(1-x,y)$ and a similar identity for $I_{4,1}(x,y)$. Further treatment of special values of Nielsen polylogarithms \big(in particular $S_{3,2}$ at elements of the weight 2 Bloch group, such as $ S_{3,2}(\phi)$ for $\phi = \frac{1 + \sqrt{5}}{2}$\big) which were obtained from this clean single-valued procedure, can be found in \cite{Charlton:2019gvp}.

\subsection{Weight 4}
We recall first an identity that ``reduces a depth 2 combination to depth 1'', which was predicted by Goncharov, with $\li_4$-terms first made explicit in \cite{Gangl_weight4}.
For the following, we note that $\licsv_4(x) = \Icsv_4(x)$, because of the symbol level identity $ \li_4(x) = I_4(x)$ modulo products.
\begin{Proposition}
	The following identity of clean single-valued functions holds
\begin{gather*}
\Icsv_{3,1}(1-x,y)+\Icsv_{3,1}(x,y) \\
\quad{} =
 \licsv_4\Big(\frac{1-x}{1-y}\Big)-\licsv_4\Big(\frac{1-y}{x}\Big)
-3 \licsv_4\Big(\frac{y}{1-x}\Big)
-3 \licsv_4\Big(\frac{y}{x}\Big)
+\licsv_4\Big(\frac{y}{y-1}\Big)\\
\quad\hphantom{=}{}
 -\frac{1}{2} \licsv_4\Big(\frac{(1-x) y}{x (1-y)}\Big)
-\frac{1}{2} \licsv_4\Big(\frac{x y}{(1-x) (1-y)}\Big)
+\frac{1}{2} \licsv_4\Big(\frac{(1-y) y}{(1-x) x}\Big) .
\end{gather*}

\begin{proof}
 As indicated above, a (slight variant of the) corresponding identity was given in \cite{Gangl_weight4} on the level of the symbol modulo products, and with $\Icsv_{3,1}$ and $\licsv_{4}$ replaced by $I_{3,1}$ and $\li_{4}$.
	Hence the difference of the two sides in the equation stated above is a constant by Theorem \ref{thm:main}. This constant must be 0 since each side vanishes on the real line, as we take the imaginary part in the definition of the functions.
\end{proof}
\end{Proposition}

Taking $x = \frac{1}{2}$, $y = {\rm i}$ (any non-real $y$ will also give a reduction) in the above identity leads~to
\begin{align*}
	2 \Icsv_{3,1}\Big(\frac{1}{2},{\rm i}\Big) = & -\licsv_4\Big(\frac{-1 + {\rm i}}{2} \Big) - 6 \licsv_4(2{\rm i}) + \licsv_4\Big(\frac{1+{\rm i}}{4} \Big)
\\
& + \licsv_4\Big(\frac{1-{\rm i}}{2} \Big)- \licsv_4(2 - 2{\rm i})+ \frac{1}{2}\licsv_4(4 + 4{\rm i}).
\end{align*}
One can apply the inversion relation $\licsv_4(x) = -\licsv_4\big(x^{-1}\big)$, and the inversion relation (with $x = \frac{1}{2}$, $y = {\rm i}$) from Proposition~\ref{prop:inversion_depth_2}, namely
\begin{gather*}
\Icsv_{3,1}(2,-{\rm i}) - \Icsv_{3,1}\Big(\frac{1}{2},{\rm i}\Big) = \licsv_4({\rm i}) - 3 \licsv_4\Big({-}\frac{\rm i}{2}\Big) - \licsv_4\Big(\frac{1}{2}\Big) ,
\end{gather*}
\big(wherein $\licsv_4\big(\frac{1}{2}\big) = 0$ as we take the imaginary part\big) to obtain
\begin{align*}
2 \Icsv_{3,1}(2,-{\rm i}) = {}& -\licsv_4\Big(\frac{-1 + {\rm i}}{2} \Big) + 2 \licsv_4({\rm i}) - \frac{1}{2} \licsv_4\Big(\frac{1-{\rm i}}{8} \Big)
\\
& + 2 \licsv_4\Big(\frac{1+{\rm i}}{4} \Big) + \licsv_4\Big(\frac{1-{\rm i}}{2}\Big) .
\end{align*}
It is possible to compute explicitly how $\Icsv_{3,1}(x,y)$ is expressed using the classical polylogarithms and iterated integrals (an implementation is available in the \texttt{PolyLogTools} package~\cite{Duhr:2019tlz}). One then obtains that
\begin{gather*}
2\Icsv_{3,1}(2, -{\rm i}) =
\Im \bigg[ 4 \li_{3,1}\Big({-}\frac{\rm i}{2}, {\rm i}\Big) + 4 \log2 \li_{2,1}\Big({-}\frac{\rm i}{2}, {\rm i}\Big) + 2\zeta_2 \li_2({\rm i})
\\ \hphantom{2\Icsv_{3,1}(2, -{\rm i}) =}
\qquad{}- \Big(2 \li_2\Big(\frac{2+{\rm i}}{5}\Big)
-2 \li_2\Big(\frac{4+2 {\rm i}}{5}\Big)+\log \Big(\frac{4+3{\rm i}}{5}\Big) \log 2 \Big) \log ^22 \bigg]
\\ \hphantom{2\Icsv_{3,1}(2, -{\rm i}) =}
{}-\frac{\pi}{8}\Big(\li_3\Big({-}\frac{1}{4}\Big)+2 \li_2\Big({-}\frac{1}{4}\Big) \log 2
+ 4 \log ^32- 2 \log 5 \log ^22 \Big) ,
\end{gather*}
with
\begin{gather*}
\li_{m_1,m_2}(x,y) = I_{m_1,m_2}\Big(\frac{1}{xy},\frac{1}{y}\Big) = \sum_{0 < n_1 < n_2} \frac{x^{n_1} y^{n_2}}{n_1^{m_1} n_2^{m_2}} .
\end{gather*}
Likewise the $\licsv_4$ combination evaluates to
\begin{align*}
	= {} & \Im\bigg[ 2\li_4\Big(\frac{1+{\rm i}}{2}\Big)
	-4 \li_4\Big(\frac{1+{\rm i}}{4}\Big) -\li_4\Big(\frac{1+{\rm i}}{8}\Big)
	-4 \li_4({\rm i})	+2 \li_4\Big(\frac{-1+{\rm i}}{2}\Big)
\\
&+ \Big(\li_3\Big(\frac{-1+{\rm i}}{2}\Big)
	-\frac{5}{2} \li_3\Big(\frac{1+{\rm i}}{8}\Big)
	-6	\li_3\Big(\frac{1+{\rm i}}{4}\Big)
	+ \li_3\Big(\frac{1+{\rm i}}{2}\Big)\Big) \log 2
\\
& + \Big(\li_2({\rm i})-\frac{9}{2} \li_2\Big(\frac{1+{\rm i}}{4}\Big)
	-\frac{25}{8}\li_2\Big(\frac{1+{\rm i}}{8}\Big)
	+\frac{1}{4} \li_2\Big(\frac{-1+{\rm i}}{2}\Big)\Big) \log ^22
	\\
	& +\frac{1}{96} \Big({-}2 {\rm i} \pi
	+106 \log \Big(\frac{4 - 3 {\rm i}}{5}\Big)+125 \log \Big(\frac{24 - 7{\rm i}}{25}\Big) \Big) \log ^32 \bigg] .
\end{align*}

By equating these two results, one obtains an explicit reduction for the value $\Im\big(\li_{3,1}\big({-}\frac{\rm i}{2}, {\rm i}\big)\big)$ in terms lower depth and products. \big(Moreover, the depth 2 term $ \li_{2,1}\big({-}\frac{\rm i}{2},{\rm i}\big)$ is expressible purely in terms of $\li_3$ and products, via the known reduction of all weight 3 MPL's to depth~1. This is essentially a consequence of \cite[Appendix A.3.5(2)]{LewinBook}, subsequently also re-established in~\cite{Kellerhals_GAFA}.)

Utilising a well-known lattice reduction algorithm (``LLL''), we can find the following simpler numerically checked reduction
\begin{gather*}
\li_{3,1}\Big({-}\frac{\rm i}{2}, {\rm i}\Big) + \log2 \li_{2,1}\Big({-}\frac{\rm i}{2},{\rm i}\Big)\\
\quad{}\overset{?}={}
 \Im\bigg[3\li_4({\rm i})+28	\li_4\Big(\frac{1+{\rm i}}{2}\Big)
-36 \li_4\Big(\frac{-1+{\rm i}}{2}\Big)
+14 \li_3\Big(\frac{1+{\rm i}}{2}\Big) \log 2\\
\qquad\qquad{}-18 \li_3\Big(\frac{-1+{\rm i}}{2}\Big) \log 2
 +\Big(\frac{5}{2} \log^22 - \frac{1}{2} \zeta_2 \Big) \li_2({\rm i})
-3\li_2\Big(\frac{-1+{\rm i}}{2}\Big) \log ^22\bigg]
\\
\qquad{} +\frac{\pi}{32} \li_3\Big({-}\frac{1}{4}\Big)+\frac{\pi}{16} \li_2\Big({-}\frac{1}{4}\Big)
\log 2-\frac{7\pi}{24} \log ^32 ,
\end{gather*}
where $\overset{?}{=}$ indicates that this is a conjectural identity checked to several hundreds of digits.
One also notices the following relation amongst the $\licsv_4$ terms above
\begin{gather*}
 37\licsv_4\Big(\frac{-1 \!+ {\rm i}}{2}\Big) - 36 \licsv_4({\rm i}) + \frac{1}{2} \licsv_4\Big(\frac{1\!-{\rm i}}{8} \Big) - 6 \licsv_4\Big(\frac{1\!+{\rm i}}{4}\Big) - 21\licsv_4\Big(\frac{1\!-{\rm i}}{2}\Big)\overset{?}{=} 0 .
\end{gather*}

\subsection{Weight 5}

An identity analogous to the one given in the previous subsection, but now in weight 5, needs four terms in depth~2 and has been given in \cite{Charlton:2019ilv}.

The corresponding clean single-valued identity is as follows.
Again, the symbol level identity $\li_5(x) = - I_5(x)$, modulo products, implies that $\licsv_5(x) = -\Icsv_5(x)$.)

\begin{Proposition}
The following identity holds for the clean single-valued functions
\begin{gather*}
 \frac{1}{2} \big(\Icsv_{4,1}(x,y) + \Icsv_{4,1}\big(x,y^{-1}\big)\big) +
\frac{1}{2} \big(\Icsv_{4,1}(1-x,y) + \Icsv_{4,1}\big(1-x,y^{-1}\big)\big)
\\ \quad
{}= \frac{1}{12} \licsv_5\Big(\frac{x^2 y}{(1-x) (1-y)^2}\Big)
+\frac{1}{12} \licsv_5\Big(\frac{(1-x)^2 y}{x (1-y)^2}\Big)
+\frac{1}{6} \licsv_5\Big(\frac{(1-x) x y^2}{y-1}\Big)
\\ \quad\hphantom{=}
{}+\frac{1}{6} \licsv_5\Big(\frac{(1-y) y}{(1-x) x}\Big)
-\frac{1}{2} \licsv_5\Big(\frac{1-x}{x (y-1)}\Big)
-\frac{1}{2} \licsv_5\Big(\frac{x y}{(1-x) (1-y)}\Big)
\\ \quad\hphantom{=}
{}-\frac{1}{2} \licsv_5\Big(\frac{(1-x) (1-y)}{-x}\Big)
-\frac{1}{2} \licsv_5\Big(\frac{(1-x) y}{x (1-y)}\Big)
\\ \quad\hphantom{=}{}
-\frac{7}{4} \licsv_5\Big(\frac{y}{1-x}\Big)
-\frac{7}{4} \licsv_5((1-x) y)
-\frac{7}{4} \licsv_5\Big(\frac{y}{x}\Big)
-\frac{7}{4} \licsv_5(x y)
\\ \quad\hphantom{=}
{}-\licsv_5\Big(\frac{1-y}{x}\Big)-\licsv_5\Big(\frac{1-x}{1-y}\Big)
-\licsv_5\Big(\frac{(1-x) y}{y-1}\Big)
-\licsv_5\Big(\frac{x y}{y-1}\Big)
\\ \quad\hphantom{=}
{}+\frac{1}{2} \licsv_5(1-x)+\frac{1}{2} \licsv_5\!\Big(\frac{1}{x}\Big)
+\licsv_5\!\Big(\frac{x-1}{x}\Big)
+\licsv_5\!\Big(\frac{1}{1-y}\Big)
+\licsv_5\!\Big(\frac{y}{y-1}\Big) - 2 \zeta_5 .
\end{gather*}

\begin{proof}
 We infer this result from the exact same identity, where $\Icsv_{4,1}$ and $\licsv_5$ have been replaced by $I_{4,1}$ and $\li_5$, respectively, and $2\zeta_5$ removed,
 which was proved (on the level of the symbol, and modulo products) in \cite{CharltonPhD} (a slight variant thereof) and \cite{Charlton:2019ilv}.
	From the previous machinery, the identity now follows, up to a constant $c$ on the right hand side. Taking $x = 1$, $y = 1$ leads to
	\begin{gather*}
		\Icsv_{4,1}(0,1) + \Icsv_{4,1}(1,1) = - 3 \licsv_5(1) + c ,
	\end{gather*}
	wherein the terms $\licsv_5(0) = 0 = \licsv_5(\infty)$ have disappeared. From the shuffle identity
	\begin{gather*}
		\Icsv_{4,1}(0,1) = \Icsv(0; 0, 0, 0, 0, 1; 1) = \Icsv(0; 1, 0, 0, 0, 0; 1) = -\licsv_5(1) = -2\zeta_5 ,
	\end{gather*}
	and the evaluations $\licsv_5(1) = -\Icsv_5(1) = 2 \zeta_5$, and $\Icsv_{4,1}(1,1) = \big({-}\binom{5}{4} - 1\big) \zeta_5 = -6 \zeta_5$ from~\eqref{eq:clean_zeta_d1}, Corollary~\ref{prop:Icsv_root_unity} and Proposition~\ref{prop:depth2_at_1}, we see
	\begin{gather*}
		c = (-2 - 6 + 6) \zeta_5 = -2 \zeta_5 ,
	\end{gather*}
	as claimed.
\end{proof}
\end{Proposition}

Now set $x = \frac{1}{2}$, $y = -1$ in this identity (and apply the inversion results $\licsv_5(x) = \licsv_5\big(x^{-1}\big)$ and $\Icsv_{4,1}(2,-1) + \Icsv_{4,1}\big(\frac{1}{2},-1\big) = -\licsv_5(-1) - 4 \licsv_5\big({-}\frac{1}{2}\big) - \licsv_5\big(\frac{1}{2}\big)$ from Proposition~\ref{prop:inversion_depth_2}), and we obtain
\begin{gather*}
\Icsv_{4,1}(2,-1) = -\frac{3}{2} \licsv_5(-1) + 5 \zeta_5 + \frac{1}{2} \licsv_5\Big({-}\frac{1}{2}\Big)
 - \frac{1}{4} \licsv_5\Big({-}\frac{1}{8}\Big)
 \\ \hphantom{\Icsv_{4,1}(2,-1) =}
 {}+ 2\licsv_5\Big(\frac{1}{4}\Big) - \frac{5}{2}\licsv_5\Big(\frac{1}{2}\Big) .
\end{gather*}
Using the duplication relation
\begin{gather*}
	\licsv_5\Big(\frac{1}{4}\Big) = 16 \licsv_5\Big({-}\frac{1}{2}\Big) + 16 \licsv_5\Big(\frac{1}{2}\Big)
\end{gather*}
and that $\licsv_5(-1) = -\frac{15}{16} \licsv_5(1)$, we can simplify this as
\begin{gather*}
\Icsv_{4,1}(2,-1) = \frac{61}{16} \zeta_5 + \frac{65}{2} \licsv_5\Big({-}\frac{1}{2}\Big) - \frac{1}{4} \licsv_5\Big({-}\frac{1}{8}\Big) + \frac{59}{2}\licsv_5\Big(\frac{1}{2}\Big) .
\end{gather*}
From the implementation in \texttt{PolyLogTools}, we find
\begin{align*}
\Icsv_{4,1}(2,-1) = {}	& 2\li_{4,1}\Big({-}\frac{1}{2},-1\Big) + 2\li_{3,1}\Big({-}\frac{1}{2},-1\Big) \log2 + \frac{6}{5} \li_{2,1}\Big({-}\frac{1}{2},-1\Big) \log^22
\\
& +2 \li_4\Big(\frac{1}{2}\Big) \log 2
-\frac{4}{5} \li_3\Big({-}\frac{1}{2}\Big) \log^22
-\frac{28}{15} \li_2\Big({-}\frac{1}{2}\Big) \log^32
\\
& +\frac{39}{20} \zeta_3 \log ^22+\frac{3}{4} \zeta_2\zeta_3-\frac{14}{15} \zeta_2 \log^32-\frac{16}{15}\log ^52
+\frac{16}{15} \log ^42 \log 3 .
\end{align*}
Likewise, the $\licsv_5$ terms evaluate as
\begin{gather*}
\frac{61}{16} \zeta_5 + \frac{65}{2} \licsv_5\Big({-}\frac{1}{2}\Big) - \frac{1}{4} \licsv_5\Big({-}\frac{1}{8}\Big) + \frac{59}{2}\licsv_5\Big(\frac{1}{2}\Big)
\\ \qquad
{}=65 \li_5\Big({-}\frac{1}{2}\Big)-\frac{5}{2}\li_5\big({-}\frac{1}{8}\Big)
+59 \li_5\Big(\frac{1}{2}\Big)+65 \li_4\Big({-}\frac{1}{2}\Big) \log 2
\\ \qquad\hphantom{=}
{}-\frac{3}{2}\li_4\Big({-}\frac{1}{8}\Big) \log 2+59 \li_4\Big(\frac{1}{2}\Big) \log 2
-\frac{27}{10} \li_3\Big({-}\frac{1}{8}\Big) \log^22
\\ \qquad\hphantom{=}
{}+39 \li_3\Big({-}\frac{1}{2}\Big) \log ^22 +\frac{52}{3}\li_2\Big({-}\frac{1}{2}\Big) \log ^32
-\frac{18}{5} \li_2\Big({-}\frac{1}{8}\Big) \log ^32
\\ \qquad\hphantom{=}
{}+\frac{61}{16}\zeta_5+\frac{1239}{40} \zeta_3 \log ^22 -\frac{59}{6} \zeta_2 \log ^32
+\frac{16}{15} \log ^42 \log 3-\frac{9}{5} \log ^52 .
\end{gather*}
Equating these two results gives an evaluation for $\li_{4,1}\big({-}\frac{1}{2},-1\big)$ in terms of products and lower depth.

Application of the lattice reduction algorithm ``LLL'' on the set of arising values (after also introducing $\zeta_4\log2$) leads to the following simpler candidate reduction for $\li_{4,1}\big({-}\frac{1}{2},-1\big)$ alone
\begin{align*}
 \li_{4,1}\Big({-}\frac{1}{2},-1\Big) \overset{?}{=} {}& -2 \li_5\Big(\frac{1}{2}\Big)
-8 \li_5\Big({-}\frac{1}{2}\Big)-2 \li_4\Big({-}\frac{1}{2}\Big) \log 2
\\
& -\frac{281}{64} \zeta_5-\frac{3}{8} \zeta_2 \zeta_3+\frac{31}{16} \zeta_4 \log 2 	+\frac{3}{8} \zeta_3 \log ^22-\frac{1}{4} \zeta_2 \log ^32+\frac{1}{24}\log ^52 ,
\end{align*}
along with a similar reduction for $\li_{3,1}\big({-}\frac{1}{2},-1\big)$
\begin{align*}
 \li_{3,1}\Big({-}\frac{1}{2},-1\Big) \overset{?}{=}{}& -3 \li_4\Big(\frac{1}{2}\Big)
-6 \li_4\Big({-}\frac{1}{2}\Big)
-2 \li_3\Big({-}\frac{1}{2}\Big) \log 2
\\
& -\frac{31}{16}\zeta_4+\frac{3}{4} \zeta_2 \log ^22-\frac{3}{4} \zeta_3 \log 2-\frac{5}{24} \log ^42 .
\end{align*}

\subsection*{Acknowledgements}
We are grateful to Falko Dulat for early collaboration on this project. We would like to thank the MITP in Mainz, where this work was started, and the HIM in Bonn and the GGI Florence for hospitality where part of this work was developed. We are particularly grateful to the organisers of the workshop on ``Modular forms, periods and scattering amplitudes'' at the ETH Z\"urich in April 2019, where some of our results had been first presented. In particular, we are grateful to Francis Brown and Erik Panzer for pointing out the relevance of the Dynkin operator in the construction of the clean single-valued analogues of multiple polylogarithms. SC is grateful to the Max-Planck-Institut f\"ur Mathematik in Bonn, for support, hospitality and excellent working conditions during his stay, where some of this work was undertaken. SC was also partially supported by DFG Eigene Stelle grant CH 2561/1-1, for Projektnummer 442093436.

\pdfbookmark[1]{References}{ref}
\LastPageEnding

\end{document}